%% file: regular.tex
\documentclass[submission,copyright,creativecommons]{eptcs}
\input{preamble.tex}

\providecommand{\event}{QPL 2015}
\title{Categories of relations as models of quantum theory}
\def\titlerunning{Categories of relations as models of quantum theory}
\def\authorrunning{C. Heunen \& S. Tull}
\author{Chris Heunen\thanks{Supported by EPSRC Fellowship EP/L002388/1.} \hspace*{1ex}and Sean Tull\thanks{Supported by EPSRC Studentship OUCL/2014/SET. We thank Bob Coecke and Aleks Kissinger for helpful discussions.}
\email{\{heunen,sean.tull\}@cs.ox.ac.uk}
\institute{University of Oxford, Department of Computer Science}}
\begin{document}
\maketitle

\begin{abstract} 
  Categories of relations over a regular category form a family of models of quantum theory.
  Using regular logic, many properties of relations over sets lift to these models, including the correspondence between Frobenius structures and internal groupoids. 
  Over compact Hausdorff spaces, this lifting gives continuous symmetric encryption.
  Over a regular Mal'cev category, this correspondence gives a characterization of categories of completely positive maps, enabling the formulation of quantum features.
  These models are closer to Hilbert spaces than relations over sets in several respects:  Heisenberg uncertainty, impossibility of broadcasting, and behavedness of rank one morphisms.
\end{abstract}

\section{Introduction}

Many features of quantum theory can be abstracted to arbitrary \emph{compact dagger categories}~\cite{abramskycoecke:categoricalsemantics}. Thus we can compare models of quantum theory in a unified setting, and look for features that distinguish the category $\FHilb$ of finite-dimensional Hilbert spaces that forms the traditional model.
The category $\cat{Rel}$ of sets and \emph{relations} is an alternative model. It exhibits many features considered typical of quantum theory~\cite{baez:quandaries}, but also refutes presumed correspondences between them~\cite{Coecke2013,heunenkissinger:cbh}. %: for example, between commutativity of observables and distributivity of yes-no measurements~\cite{Coecke2013}, between commutativity of observables and broadcasting, between signalling and kinematic dependence, and between bit commitment and entanglement~\cite{heunenkissinger:cbh}. 
However, apart from $\Rel$ and its subcategory corresponding to Spekkens' toy model~\cite{coeckeedwards:spek}, few alternative models have been studied in detail.

We consider a new family of models by generalizing to categories $\RelC$ of relations over an arbitrary \emph{regular category} $\catC$, including % any abelian category, any algebraic variety, and any topos.
any algebraic category like that of groups, any abelian category like that of vector spaces, and any topos like that of sets. 
Despite this generality, \emph{internal logic} allows us to state and prove results as if working in $\Rel$, as reviewed in Section~\ref{sec:regular}. 
Just as $\Rel$ is of independent interest, also
$\RelC$ is not just a toy model: we shortly discuss \emph{continuous symmetric encryption} by letting $\catC$ consist of compact Hausdorff spaces. For another example see~\cite{bonchietal:signalflow}.

Section~\ref{sec:properties} shows that $\RelC$ often has the unusual feature of lacking discernible measurement outcomes, but when $\catC$ is a regular \emph{Mal'cev category} it behaves more like $\FHilb$ than $\Rel$ in three ways: the \emph{Heisenberg uncertainty} principle, 
the ability to \emph{broadcast}, and behavedness of \emph{rank one} morphisms. 
Thus this family of models is genuinely different, and leads to the following scale.
%However, $\RelC$ differs in some major examples by lacking discernible measurement outcomes.
\begin{center}\begin{tabular}{c@{\qquad\qquad}c@{\qquad\qquad}c}
  ``least quantum'' & $\leftrightarrow$& ``most quantum'' \\
  $\cat{Rel}$ & $\cat{Rel}(\cat{C})$ & $\cat{FHilb}$
\end{tabular}\end{center}
%However, we also find that $\RelC$ often differs from $\cat{FHilb}$ by lacking discernible measurement outcomes.
%\\
These properties are stated using the CP construction~\cite{Coecke2013i,heunenkissingervicary:cp}, that makes the Frobenius structures in a category into the objects of a new one. The objects of $\CP(\Rel)$ are \emph{groupoids}~\cite{Heunen2013}, and Section~\ref{sec:cp} proves that this holds for $\RelC$ too, for any regular $\catC$.
Our main result, discussed in Section~\ref{sec:malcev}, is that
\begin{equation}\label{eq:maintheorem}
  \CP(\RelC) \simeq \Rel(\cat{Cat}(\cat{C}))
\end{equation}
for regular Mal'cev categories $\cat{C}$, whose internal categories $\cat{Cat}(\cat{C})$ and groupoids are understood.
Thus we link the CP passage, from state spaces to algebras of observables, to \emph{categorification}.% in a large class of categories, including all abelian categories such as modules over a ring, and all algebraic varieties like the category of groups.

\paragraph{Notation}

We briefly recall the \emph{graphical calculus} for monoidal dagger categories; for more detail, see~\cite{selinger:graphicallanguages}. Morphisms $f$ are drawn as
$\setlength\morphismheight{3mm}\begin{pic}
  \node[morphism,font=\tiny] (f) at (0,0) {$f$};
  \draw (f.south) to +(0,-.1);
  \draw (f.north) to +(0,.1);
\end{pic}$; composition, tensor product, and dagger, as:
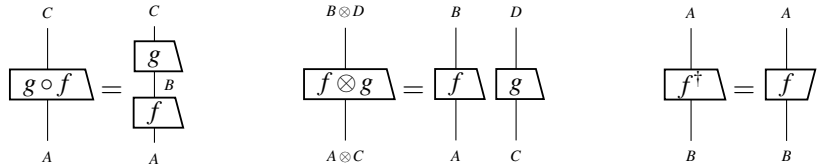
\[
  \begin{pic}
    \node[morphism] (f) {$g \circ f$};
    \draw (f.south) to +(0,-.55) node[below] {$A$};
    \draw (f.north) to +(0,.55) node[above] {$C$};
  \end{pic}
  = 
  \begin{pic}
    \node[morphism] (g) at (0,.75) {$g\vphantom{f}$};
    \node[morphism] (f) at (0,0) {$f$};
    \draw (f.south) to +(0,-.2) node[below] {$A$};
    \draw (g.south) to node[right] {$B$} (f.north);
    \draw (g.north) to +(0,.2) node[above] {$C$};
  \end{pic}
  \qquad\qquad
  \begin{pic}
    \node[morphism] (f) {$f \otimes g$};
    \draw (f.south) to +(0,-.55) node[below] {$A \otimes C$};
    \draw (f.north) to +(0,.55) node[above] {$B \otimes D$};
  \end{pic}
  = 
  \begin{pic}
    \node[morphism] (f) at (-.4,0) {$f$};
    \node[morphism] (g) at (.4,0) {$g\vphantom{f}$};
    \draw (f.south) to +(0,-.55) node[below] {$A$};
    \draw (f.north) to +(0,.55) node[above] {$B$};
    \draw (g.south) to +(0,-.55) node[below] {$C$};
    \draw (g.north) to +(0,.55) node[above] {$D$};
  \end{pic}
  \qquad\qquad
  \begin{pic}
    \node[morphism] (f) {$f^\dag$};
    \draw (f.south) to +(0,-.55) node[below] {$B$};
    \draw (f.north) to +(0,.55) node[above] {$A$};
  \end{pic}
  =
  \begin{pic}
    \node[morphism,hflip] (f) {$f$};
    \draw (f.south) to +(0,-.55) node[below] {$B$};
    \draw (f.north) to +(0,.55) node[above] {$A$};
  \end{pic}
\]
Distinguished morphisms are depicted with special diagrams: the identity $A \to A$ is just the line, $\begin{aligned}\begin{pic}
  \draw (0,-.15) to (0,.15);
\end{pic}\end{aligned}$\;; the (identity on) the monoidal unit object $I$ is the empty picture, and the swap map of symmetric monoidal categories becomes $\begin{aligned}\begin{pic}[scale=.25]
  \draw (0,-.5) to[out=80,in=-100] (1,.5);
  \draw (1,-.5) to[out=100,in=-80] (0,.5);
\end{pic}\end{aligned}$.
In particular, we will draw $\tinymult \colon A \otimes A \to A$ and $\tinyunit \colon I \to A$ for the multiplication and unit of \emph{Frobenius structures}, made precise in Definition~\ref{def:frobeniusstructure} below.

\paragraph{Continuous symmetric encryption}

We now illustrate the utility of $\RelC$ in its own right, not as a toy model. The case $\cat{C}=\cat{Set}$ can model symmetric encryption~\cite{Stay2013}, and the same techniques apply to any regular $\cat{C}$.
A \emph{symmetric encryption protocol} in $\RelC$ is specified by an \emph{encryption} morphism $E \colon M \times K \to E$ 
relating \emph{plaintext} $P$ and \emph{key} $K$ to \emph{ciphertext} $C$, satisfying:
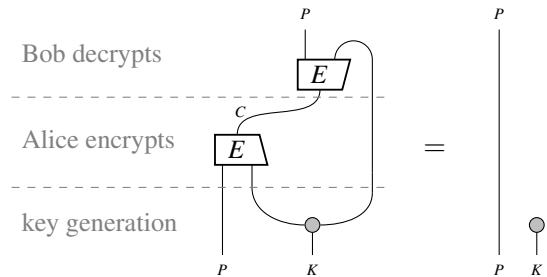
\[ 
  \begin{pic}
    \node[morphism](e) at (0,1.5){$E$};
    \node[morphism, hflip](d) at (1.1,2.5){$E$};
    \node[dot](c) at (1,0.5){};
	\draw([xshift = 1mm]d.north east) to[out=90, in=180] +(0.25,+0.25) to[out = 0, in = 90] +(0.25,-0.25) to +(0,-1.5) node[below](g){};
	\draw (g.north) to (g.south);
	\draw (c.east) to[out = 0, in = -90] (g.south);
	\draw([xshift=1mm]e.south east) to +(0,-0.3) node[below](h2){};
	\draw (h2.north) to [in = 180, out = -90](c.west);
	\draw (1,0.1) node[below]{$K$} to (c.south);
	\draw([xshift=-1mm]e.south west) to +(0,-1.2) node[below]{$P$};
	\draw ([xshift=-1mm]d.north west) to +(0,0.4) node[above]{$P$};
  	\draw (e.north) to[out = 90, in = -90] node[left=3mm]{$C$} (d.south);
	\draw[dashed,gray] (-3, 1) to (2,1);
	\draw[dashed,gray] (-3, 2.2) to (2,2.2);
	\node[gray,anchor=west] at (-3, 0.5){\small key generation};
	\node[gray,anchor=west] at (-3, 1.6){\small Alice encrypts};
	\node[gray,anchor=west] at (-3, 2.75){\small Bob decrypts};
  \end{pic}
  \quad = \quad
  \begin{pic}
	\draw(0,0) node[below]{$P$} to +(0,3) node[above]{$P$};
	\draw(0.5,0) node[below]{$K$} to +(0,.4) node[dot]{};
  \end{pic}
\]
The Frobenius structure here is the canonical copying and deleting that is available in $\RelC$~\cite{carboniwalters:cartesianbicategories}.
Equivalently, $(\forall p \colon P,\, k \colon K)(\exists c \colon C) E(p,k,c)$ and $E(p,k,c) \wedge E(p',k,c) \Rightarrow p=p'$.  The protocol is \emph{secure} when no information about plaintext can be deduced from ciphertext without the key:
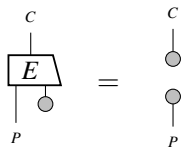
\[
  \begin{pic}
    \node[morphism](e) at (0,0) {$E$};
    \draw([xshift = -1mm]e.south west) to +(0,-0.5) node[below]{$P$};
	\draw([xshift = 1mm]e.south east) to +(0,-0.25) node[dot]{};
	\draw(e.north) to +(0,0.3) node[above]{$C$};
  \end{pic}
  \quad = \quad
  \begin{pic}
	\draw (0,0) node[below]{$P$} to (0,0.5);
	\draw (0,1) to (0,1.3) node[above]{$C$};
	\node[dot](b) at (0,0.4){};
	\node[dot](t) at (0,.9){};
  \end{pic}
  % \qquad
  % (\forall p:P, \ c:C) \  (\exists k: K) \ E(p,k,c)
\]

For example, \emph{one-time pad encryption} is the following secure encryption protocol in $\Rel$: take $P$, $K$ and $C$ to be the set $\{f \colon [1,\ldots,n] \to G\}$ of messages in a given group $G$ of length $n$, and take $E$ to be the function $E(p,k)(t) = p(t)k(t)$~\cite{Stay2013}.
A \emph{continuous} version can be described in the category $\Rel(\cat{KHaus})$ of relations over compact Hausdorff spaces. An analogue signal is a continuous function from an interval $[0,T] \subseteq \mathbb{R}$ to the unit circle $S^1 \subseteq \mathbb{C}$.  Take $P$, $K$ and $C$ to be the space of such signals (under the product topology) and set $E(p,k)(t) = p(t)k(t)$.  
%\item An encoding relation $E$ in $\cat{Grp}$ must preserve the group structure in each of its arguments. This resembles `partially homomorphic' encryption techniques such as ElGamal \cite[Figure 3]{Fontaine2007}, only now $E$ is also homomorphic in its private key argument.
The latter protocol is useful when encrypting \textit{e.g.}\ (continuous) speech rather than (discrete) strings of text, and was proven secure by Shannon himself~\cite{Shannon1943}. \footnote{Caveat: the latter protocol requires arbitrary key signals. Practically, Alice could sample her signal at some points, and then transmit using (discrete) one-time pad encryption; Bob can then reconstruct an approximation of the signal. Shannon describes how to generate noise that, when sampled at a given frequency, provides a secure key.
Thus the former protocol approximates the latter to any resolution.}

\paragraph{Future work}

This work is in its early stages, and opens many directions of investigation. 
\begin{itemize}
\item The \emph{proof technique} of internal (regular) logic is very useful. We intend to develop a \emph{graphical} version, where wires in string diagrams can be annotated with `elements', and investigate its expressiveness~\cite{selinger:complete}.

\item The correspondence~\eqref{eq:maintheorem} between groupoids and operator algebra remains to be understood, perhaps through so-called \emph{groupoidification}~\cite{morton:twovectorspaces}. If $\cat{C}$ is \emph{exact} Mal'cev, so is $\cat{Cat}(\cat{C})$~\cite[3.2]{Gran1999}. Hence $\CP(\Rel(\cat{Cat}^n(\cat{C}))) \simeq \Rel(\cat{Cat}^{n+1}(\cat{C}))$, leading to \emph{higher categories}. 

\item Internal categories in the category of groups are \emph{crossed modules}~\cite{brown:groupoids}. Tools from categorical quantum mechanics might shed light on crossed modules, or vice versa. In general, standard notions from categorical quantum mechanics should be investigated in $\RelC$, such as (strong) \emph{complementarity} of Frobenius structures~\cite{coeckeduncan:zx}. 

\item Relation-like categories have been axiomatized: as \emph{allegories} by Freyd and Scedrov \cite{freyd1990categories}, and as \emph{bicategories of relations} by Carboni and Walters \cite{ carboniwalters:cartesianbicategories}. We hope to extend our study to these more general settings.
\end{itemize}

\section{Categories of relations and regular logic}\label{sec:regular}

This section describes regular categories, and their internal regular logic, by way of example; for more information, see~\cite{Butz1998}. We will be led by the construction of $\Rel$ from the category $\cat{Set}$ of sets and functions.
Of course, both categories have the same objects. The issue is how to describe relations and their composition in terms of functions. 

Observe that a relation $R \subseteq A \times B$ induces a pair of functions $R_1 \colon R \to A$ and $R_2 \colon R \to B$, namely $(a,b) \mapsto a$ and $(a,b) \mapsto b$. Moreover, the inclusion $R \hookrightarrow A \times B$ is monic. Equivalently, the two functions $R_1,R_2$ are \emph{jointly monic}, and conversely, any two jointly monic functions $\smash{A \stackrel{R_1}{\leftarrow} R \stackrel{R_2}{\rightarrow} B}$ determine a relation $R \subseteq A \times B$. Thus we can describe relations in terms of functions. Composition of relations is given in these terms by pullback:
\begin{equation}\label{eq:pullbackcomposition}
\begin{pic}[font=\small,xscale=2]
  \node (x) at (-2,0) {$A$};
  \node (y) at (0,0) {$B$};
  \node (z) at (2,0) {$C$};
  \node (r) at (-1,1) {$R$};
  \node (s) at (1,1) {$S$};
  \node (p) at (0,2) {$R \times_B S$};
  \draw[->,font=\tiny] (r) to node[left=1.5mm]{$R_1$} (x);
  \draw[->,font=\tiny] (r) to node[right=1mm]{$R_2$} (y);
  \draw[->,font=\tiny] (s) to node[left=1.5mm]{$S_1$} (y);
  \draw[->,font=\tiny] (s) to node[right=1mm]{$S_2$} (z);
  \draw[->] (p) to (r);
  \draw[->] (p) to (s);
 \node (sr) at (0,1) {$S \circ R$};
 \draw[->>, dotted] (p) to (sr);
 \draw[->,font=\tiny,dotted] (sr) to (x);%node[above=-3mm]{$(S \circ R)_1$} (x);
 \draw[->,font=\tiny,dotted] (sr) to (z);%node[above=-3mm]{$(S \circ R)_2$} (z);
\end{pic}\end{equation}
The pullback itself is not good enough, as $R \times_B S \to A \times C$ might not be monic. To ensure this, we have to factor that function as a surjection followed by an injection, and consider its image $S \circ R$. Because pullbacks of surjections are surjections, the subobject $S \circ R$ is unique, giving a well-defined category $\Rel$.

All in all, we have used the following properties of $\cat{Set}$: it has products, pullbacks, a way to factorize morphisms as injections after surjections via their images, and stability of surjections under pullback. 
Generalizing surjections to regular epimorphisms leads to regular categories. 

\begin{definition}
  An epimorphism is \emph{regular} when it is a coequalizer. 
  The \emph{kernel pair} of a morphism $f \colon A \to B$ is a pullback cone of $A \smash{\stackrel{f}{\rightarrow}} B \smash{\stackrel{f}{\leftarrow}} A$.
  A category is \emph{regular} when it has finite limits, coequalizers of kernel pairs, and regular epimorphisms are stable under pullback.
\end{definition}

\begin{example}
  Examples of regular categories abound: 
  any topos, such as $\cat{Set}$;
  any algebraic variety, such as the categories $\cat{Gp}$ of groups,  $\cat{Rng}$ of rings, or $\cat{Vect}_k$ of vector spaces; 
  any category monadic over $\cat{Set}$, such as $\cat{KHaus}$;
  any abelian category, such as the category $\cat{Mod}_R$ of modules over a ring;
  any bounded meet-semilattice, considered as a category;
  any functor category $[\cat{C},\cat{D}]$ to regular $\cat{D}$.
\end{example}

There is another way to consider the construction~\eqref{eq:pullbackcomposition}, namely by \emph{regular logic}.
This is the fragment of first order logic whose formulae use only the connectives $\exists$ and $\wedge$ and equality. 
In $\cat{Set}$, we can describe~\eqref{eq:pullbackcomposition} using a regular formula as
\begin{equation}\label{eq:regularcomposition}
  S \circ R = \{ (a,c) \in A \times C \mid (\exists b \in B)\; R(a,b) \wedge S(b,c) \}.
\end{equation}
This makes sense in any regular category $\cat{C}$: any regular formula $\phi$ whose function symbols are morphisms in $\cat{C}$ and whose relation symbols are subobjects in $\cat{C}$ inductively defines a subobject $\sem{ (a_1,\ldots,a_n) \in A_1 \times \ldots \times A_n \mid \phi} \rightarrowtail A_1 \times \cdots \times A_n$ as follows. Equality $\sem{a \in A \mid f(a)=g(a)}$ is interpreted as the equalizer of $f,g \colon A \to B$. Conjunction $\sem{a \in A \mid R(a) \wedge S(a)}$ is interpreted as the pullback of $R,S \rightarrowtail A$. Existential quantification $\sem{a \in A \mid (\exists b \in B)\, R(a,b)}$ is interpreted as the image %$\exists_A(R)$ 
of $R \rightarrowtail A \times B \smash{\stackrel{\pi_1}{\to}} A$.
This gives two equivalent ways to define our main object of study.

\begin{definition}
  Let $\cat{C}$ be a regular category. Its category $\RelC$ of relations has the same objects as $\cat{C}$, and subobjects $R \rightarrowtail A \times B$ as morphisms $A \to B$, with diagonal maps $A \to A \times A$ as identities, under composition~\eqref{eq:pullbackcomposition}, or equivalently $S \circ R = \sem{(a,c) \in A \times C \mid (\exists b \in B)\, R(a,b) \wedge S(b,c)}$.
\end{definition}

We denote morphisms in $\RelC$ as $R \colon A \relto B$, and the corresponding subobject in $\catC$ as $R \rightarrowtail A \times B$. The category $\RelC$ is a compact dagger category with the product $\times$ of $\cat{C}$ inducing $\otimes$ in $\RelC$ (where it is no longer a cartesian product), and $R^\dag(b,a) \Leftrightarrow R(a,b)$. Every object $A$ is self-dual with canonical cup $\sem{(a,b) \in A\times A \mid a=b}$, and swap maps defined similarly.

\begin{example} $\Rel(\cat{Set})$ is $\Rel$, so this is a genuine generalization; $\Rel(\cat{Gp})$ has subgroups $R \leq G \times H$ as morphisms $G \to H$; $\Rel(\cat{Vect}_k)$ has subspaces $K \leq V \oplus W$ as morphisms $V \to W$ (\textit{cf.}~\cite{bonchietal:signalflow}).
\end{example}

Whenever one can derive an implication $\phi \Rightarrow \psi$ in regular logic, it follows that $\sem{\phi} \leq \sem{\psi}$ as subobjects. This allows us to state and prove (regular) theorems in $\RelC$ as if reasoning in $\Rel$. 
The following lemma works out an example of this technique; the rest of this paper will not be so painstakingly precise.  
As in $\cat{Set}$, a relation $R \colon A \relto A$ is called \emph{symmetric} when $R(a,b) \iff R(b,a)$,
\emph{reflexive} when $\sem{a \in A \mid R(a,a)} = A$, and \emph{transitive} when $R(a,b) \wedge R(b,c) \Rightarrow R(a,c)$, equivalently $R \circ R \leq R$. A symmetric, reflexive and transitive relation is called an \emph{equivalence relation}.
As in any dagger category, a relation is called \emph{positive} when it is of the form $S^{\dagger} \circ S$ for some relation $S$.

% \begin{lemma} Any equivalence relation $R: A \relto A $ in a regular category satisfies $R = R^\dag \circ R$. \label{lem-eqrelpos}
% \begin{proof}
% From symmetry we have that $R^\dag = R$, and from transitivity that $R \circ R \leq R$. Conversely,
% \begin{align*}
% R & = \id{A \times A} \wedge R & \text{(pullback along identity)} \\
% & = \sem{(a,b) \in A \times A \mid R(a,a)} \wedge \sem{ (a,b) \in A \times A \mid R(a,b)} &{(R \text{ is reflexive})} \\ 
% & = \sem{ (a,b) \in A \times A \mid (\exists c \in A)\  R(a,c) \wedge R(c,b) \wedge (a = c)} & (\text{behaviour of equality and } \exists) \\
% & \leq \sem{ (a,b) \in A \times A \mid (\exists c \in A) \ R(a,c) \wedge R(c,b)} = R \circ R & (\exists \text{ preserves the order of subobjects})
% \end{align*}
% where $S \wedge S'$ denotes the pullback of subobjects $S$ and $S'$. 
% \end{proof}
% \end{lemma}

\begin{lemma} \label{lem-posrel}
  Any positive relation $R \colon A \relto A$ in $\catC$ regular is symmetric and satisfies $R(a,b) \Rightarrow R(a,a)$.% \wedge R(b,a)$. 
\end{lemma}
\begin{proof}
  Since $R$ is positive we have that $R = S^\dagger \circ S$, for some relation $S: A \relto B$. Then $R$ is symmetric since $R = S^\dagger \circ S = (S^\dagger \circ S)^\dagger = R^\dagger$. Further,
  \begin{align*}
    R = S^{\dagger} \circ S %& \text{($R$ is positive)} %\\
    % &= \sem{ (a,b) \in A \times A \mid (\exists c \in B) \ S(a,c) \wedge S^{\dagger}(c,b) } & \\ 
    &= \sem{ (a,b) \in A \times A \mid (\exists c \in B) \ S(a,c) \wedge S(b,c) } & \text{(definition of $\circ$ and $\dagger$)}\\ 
    & \leq \sem{ (a,b) \in A \times A \mid (\exists c \in B) \ S(a,c)} &  \text{($\exists$ preserves the order of subobjects)}\\
    &= \sem{ (a,b) \in A \times A \mid R(a,a) } &
  \end{align*}
  demonstrating that $R(a,b) \Rightarrow R(a,a)$. 
\end{proof}

\section{Groupoids and completely positive maps}\label{sec:cp}

Frobenius structures play a central role in categorical quantum mechanics, representing C*-algebras in $\cat{FHilb}$~\cite{vicary:quantumalgebras}. In $\Rel$, they represent groupoids~\cite{Heunen2013}. Theorem~\ref{Thm-OperatorAreGroupoids} below generalizes this to $\RelC$ for any regular $\cat{C}$, by noting that the proof can be stated in regular logic.

\begin{definition}\label{def:frobeniusstructure}
  A \emph{special dagger Frobenius structure} in a dagger monoidal category is an object $A$ with morphisms $\tinymult \colon A \otimes A \to A$ and $\tinyunit \colon I \to A$ satisfying \emph{unitality}, \emph{associativity}, \emph{speciality}, and the \emph{Frobenius law}: 
  
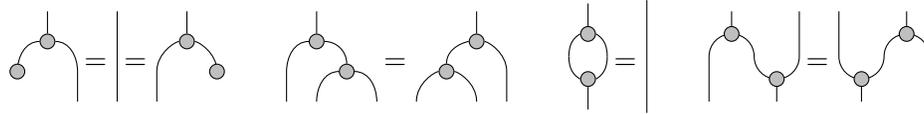
\[ 
    \begin{pic}[scale=.4] 
      \node[dot] (d) {};
  	  \draw (d) to +(0,1);
      \draw (d) to[out=0,in=90] +(1,-1) to +(0,-1);
      \draw (d) to[out=180,in=90] +(-1,-1) node[dot] {};
    \end{pic}
    =
    \begin{pic}[scale=.4]
      \draw (0,0) to (0,3);
    \end{pic}
    =
    \begin{pic}[yscale=.4,xscale=-.4]
      \node[dot] (d) {};
      \draw (d) to +(0,1);
      \draw (d) to[out=0,in=90] +(1,-1) to +(0,-1);
      \draw (d) to[out=180,in=90] +(-1,-1) node[dot] {};
    \end{pic}
  \qquad
      \begin{pic}[scale=.4]
      \node[dot] (t) at (0,1) {};
      \node[dot] (b) at (1,0) {};
      \draw (t) to +(0,1);
      \draw (t) to[out=0,in=90] (b);
      \draw (t) to[out=180,in=90] (-1,0) to (-1,-1);
      \draw (b) to[out=180,in=90] (0,-1);
      \draw (b) to[out=0,in=90] (2,-1);
    \end{pic}
    =
    \begin{pic}[yscale=.4,xscale=-.4]
      \node[dot] (t) at (0,1) {};
      \node[dot] (b) at (1,0) {};
      \draw (t) to +(0,1);
      \draw (t) to[out=0,in=90] (b);
      \draw (t) to[out=180,in=90] (-1,0) to (-1,-1);
      \draw (b) to[out=180,in=90] (0,-1);
      \draw (b) to[out=0,in=90] (2,-1);
    \end{pic}
    \qquad
    \begin{pic}
      \node[dot] (t) at (0,.6) {};
      \node[dot] (b) at (0,0) {};
      \draw (b.south) to +(0,-0.3);
      \draw (t.north) to +(0,+0.3);
      \draw (b.east) to[out=30,in=-90] +(.15,.3);
      \draw (t.east) to[out=-30,in=90] +(.15,-.3);
      \draw (b.west) to[out=-210,in=-90] +(-.15,.3);
      \draw (t.west) to[out=210,in=90] +(-.15,-.3);
    \end{pic}
    =
    \begin{pic}
    	\draw(0,0) to (0,1.5);
    \end{pic}
  \qquad
    \begin{pic}[scale =.6]
      \draw (0,0) to (0,1) to[out=90,in=180] (.5,1.5) to (.5,2);
      \draw (.5,1.5) to[out=0,in=90] (1,1) to[out=-90,in=180] (1.5,.5) to (1.5,0);
      \draw (1.5,.5) to[out=0,in=-90] (2,1) to (2,2);
      \node[dot] at (.5,1.5) {};
      \node[dot] at (1.5,.5) {};
    \end{pic}
    =
    \begin{pic}[scale = 0.6, xscale=-1]
      \draw (0,0) to (0,1) to[out=90,in=180] (.5,1.5) to (.5,2);
      \draw (.5,1.5) to[out=0,in=90] (1,1) to[out=-90,in=180] (1.5,.5) to (1.5,0);
      \draw (1.5,.5) to[out=0,in=-90] (2,1) to (2,2);
      \node[dot] at (.5,1.5) {};
      \node[dot] at (1.5,.5) {};
    \end{pic}
  \]
  It is \emph{commutative} when $\tinymult \circ \begin{aligned}\begin{pic}[scale=.25]
    \draw (0,-.5) to[out=80,in=-100] (1,.5);
    \draw (1,-.5) to[out=100,in=-80] (0,.5);
  \end{pic}\end{aligned} = \tinymult$.
\end{definition}

\begin{definition} 
  An \emph{internal category} in a finitely complete category consists of
  objects $C_0$ (objects) and $C_1$ (morphisms), and morphisms $s$ (source), $t$ (target), $u$ (identity), and $m$ (composition):
  \[\begin{pic}[font=\small,xscale=2]
    \node (0) at (-1,0) {$C_0$};
    \node (1) at (0,0) {$C_1$};
    \node (11) at (1,0) {$C_1 \times_{C_0} C_1$};
    \draw[->] (11) to node[above=-1mm] {$m$} (1);
    \draw[->] (0) to node[above=-1mm] {$u$} (1);
    \draw[->] ([yshift=-3mm]1.west) to node[above=-1mm] {$s$} ([yshift=-3mm]0.east);
    \draw[->] ([yshift=3mm]1.west) to node[above=-1mm] {$t$} ([yshift=3mm]0.east);
    % \draw[->] (1.110) to[out=110,in=180] +(.05,.3) to[out=0,in=45] node[right]{$i$} (1.70);
  \end{pic}\]
  Here $C_1 \times_{C_0} C_1$ is the pullback of $s$ and $t$. These morphisms must satisfy familiar equations representing associativity of composition and usual behaviour of identities. An \emph{internal functor} between internal categories is a pair of morphisms $(f_0, f_1)$ commuting with the above structure. 
  An \emph{internal groupoid} additionally has an inversion morphism $i \colon C_1 \to C_1$ satisfying usual axioms. We write $\cat{Cat}(\catC)$ and $\cat{Gpd}(\catC)$ for the categories of internal categories and groupoids in $\catC$.
\end{definition}

\begin{example} 
  Internal categories or groupoids in $\cat{Set}$ are just (small) categories or groupoids.
  In $\cat{Vect}_k$, internal categories are the same as internal groupoids (see Section~\ref{sec:malcev}), and are also called \emph{2-vector spaces}~\cite[Section~3]{Baez2004}.
  Internal categories in $\cat{Gp}$ are studied under the names \emph{strict 2-groups}~\cite{Baez2004a} and \emph{crossed modules}~\cite[Section~3.3]{Noohi2007}.
\end{example}

\begin{theorem} \label{Thm-OperatorAreGroupoids} 
  For any regular category $\catC$, special dagger Frobenius structures in $\Rel(\catC)$ are the same as internal groupoids in $\catC$. 
  More precisely, a special dagger Frobenius structure $(A, \tinymult, \tinyunit)$ in $\Rel(\catC)$ defines an internal groupoid in $\catC$ with composition $\tinymult$ and identities given by $\tinyunit$.
  Conversely, an internal groupoid $(C_0, C_1, m,s,t,u,i)$ in $\catC$ defines a special dagger Frobenius structure $(A, \tinymult, \tinyunit)$ in $\Rel(\catC)$ with $A=C_1$, $\tinyunit = (u \colon C_0 \rightarrowtail A)$ and $\tinymult = (m \colon A \times A \relto A)$.
\end{theorem}
\begin{proof}
  Observe that the proof for the case $\catC = \cat{Set}$ (see~\cite[Theorems~7 and~12]{Heunen2013}) can be carried out entirely in regular logic. For details, see Appendix~\ref{appendix:OpStructsAreGroupoids}.
\end{proof}

\begin{example} \label{ex-indiscgpd} For any object $A$ of $\RelC$, there is a canonical   special dagger Frobenius structure $ { (A \times A,
\begin{pic}[dotpic,scale=0.33]
% \node[circle, draw=black, scale=0.7] at (-7.5,0){\large $d$};
		\node [style=none] (0) at (-1.25, -1.25) {};
		\node [style=none] (1) at (-5.75, -1.25) {};
		\node [style=none] (2) at (-2.75, 1.25) {};
		\node [style=none] (3) at (-4.25, 1.25) {};

		\node [style=none] (6) at (-4.25, -1.25) {};
		\node [style=none] (7) at (-2.75, -1.25) {};

		\draw [in=-90, out=90] (0.center) to (2.center);
		\draw [in=90, out=-90] (3.center) to (1.center);
		\draw [in=90, out=90, looseness=2.00] (6.center) to (7.center);
\end{pic}
,
\begin{pic}[dotpic,scale=0.33]
% \node[circle, draw=black, scale=0.4] at (2,0){\huge $d^{-1}$};
		\node [style=none] (4) at (3.75, 0.25) {};
		\node [style=none] (5) at (5.25, 0.25) {};
		\draw [in=-90, out=-90, looseness=2.00] (4.center) to (5.center);	
\end{pic}
) }$ which corresponds to the \emph{indiscrete groupoid}  
  \[\begin{pic}[font=\small,xscale=2.5]
    \node (0) at (-1,0) {$A$};
    \node (1) at (0,0) {$A \times A$};
    \node (11) at (1,0) {$A \times A \times A$};
    % \draw[->] ([xshift=0.5mm]1.north) to[out = 90, in=0] +(-0.5mm,0.5) node[above]{$\sigma$} to [out = 180, in = 90] ([xshift=-0.5mm]1.north);
    \draw[->] (11) to node[above=-1mm] {$\pi_{1,3}$} (1);
    \draw[->] (0) to node[above=-1mm] {$\Delta$} (1);
    \draw[->] ([yshift=-3mm]1.west) to node[above=-1mm] {$\pi_1$} ([yshift=-3mm]0.east);
    \draw[->] ([yshift=3mm]1.west) to node[above=-1mm] {$\pi_2$} ([yshift=3mm]0.east);
    \draw[->] (1.110) to[out=110,in=180] +(.05,.3) to[out=0,in=45] node[right=-1mm]{$\sigma$} (1.70);
  \end{pic}\]
  on $A$, having a unique morphism from $a$ to $b$ for each pair $(a,b) \in A \times A$. The identities are given by the diagonal $\Delta = \langle \id{A}, \id{A} \rangle \colon A \to A \times A$, while the inversion is the swap $\sigma = \langle \pi_2, \pi_1 \rangle \colon A \times A \to A \times A$.
\end{example}

Dagger Frobenius structures in $\cat{C}$ form the objects of a new category $\CP(\cat{C})$~\cite{Coecke2013i,heunenkissingervicary:cp}. The key example is that $\CP(\cat{FHilb})$ is the category of finite-dimensional C*-algebras and completely positive maps. We briefly recall the relevant form of this \emph{CP construction}; see Appendix~\ref{appendix:cp} for a proof that this is indeed a well-defined category.
Recall that \emph{positive} (endo)morphisms in a dagger category are those the form $g^\dag \circ g$ for some morphism $g$. 

\begin{definition}\label{def:cp}
  Let $\catC$ be a compact dagger category. Then $\CP(\catC)$ is the category whose objects are special dagger Frobenius structures $(A, \tinymult)$ in $\catC$, and where morphisms $(A, \tinymult[white_dot]) \sxto{f} (B, \tinymult[dot])$ are morphisms $A \sxto{f} B$ in $\catC$ whose \emph{Choi matrix} is positive: 
  \[
    \begin{pic}[xscale=-1]
     \draw (0,.2) to (0,1) to[out=90,in=180] (.5,1.5) to (.5,1.8);
     \draw (.5,1.5) to[out=0,in=90] (1,1) to[out=-90,in=180] (1.5,.5) to (1.5,.2);
     \draw (1.5,.5) to[out=0,in=-90] (2,1) to (2,1.8);
     \node[morphism] at (1,1) {$f$};
     \node[dot] at (.5,1.5) {};
     \node[white_dot] at (1.5,.5) {};
    \end{pic}
    \quad = \quad
    \begin{pic}
     \node[morphism, hflip] (g) at (0,1) {$g$};
     \node[morphism] (f) at (0,.4) {$g$};
     \draw ([xshift=1mm]f.south east) to +(0,-.3);
     \draw ([xshift=-1mm]f.south west) to +(0,-.3);
     \draw (g.south) to (f.north);
     \draw ([xshift=1mm]g.north east) to +(0,.3);
     \draw ([xshift=-1mm]g.north west) to +(0,.3);
    \end{pic}
  \]
  for some $g \colon A \otimes B \to X$ in $\catC$. 
  Such morphisms $f$ are called \emph{completely positive}.
\end{definition}

To classify completely positive morphisms in $\RelC$, we need to identify when a relation is positive, \textit{i.e.}\ of the form $R = \sem{(a, c) \in A\times A \mid (\exists b \in B)\, S(a,b) \wedge S(c,b) }$ for some relation $S \colon A \relto B$. Lemma~\ref{lem-posrel} showed that positive relations in any regular category satisfy
\begin{equation} \label{posCondn}
  R(a,b) \Rightarrow R(a,a) \wedge R(b,a).
\end{equation}  
It follows that completely positive morphisms $R \colon (A,\tinymult[white_dot],\tinyunit[white_dot]) \to (B,\tinymult[black_dot],\tinyunit[black_dot])$ \emph{respect inverses}~\cite[7.2]{Coecke2013i}:
\begin{equation} \label{CPRelCondn}
  R(a,b) \Rightarrow R(a^{-1}, b^{-1}) \wedge R(\id{\text{dom}(a)}, \id{\text{dom}(b)}),
\end{equation} 
where $(A,\tinymult[white_dot],\tinyunit[white_dot])$ and $(B,\tinymult[black_dot],\tinyunit[black_dot])$ are regarded as internal groupoids in $\catC$. However, the converse of either statement need not hold.

\begin{proposition} \label{posRegProp} 
  The following are equivalent for a regular category:
  \begin{enumerate}[(a)]\itemsep0em 
	\item every reflexive symmetric relation is positive;
	\item a relation is positive if and only if it satisfies~\eqref{posCondn};
	\item a relation between dagger Frobenius structures is completely positive iff it respects inverses.
\end{enumerate}
\end{proposition}
\begin{proof} 
 For (a $\Rightarrow$ b), let $R: A \relto A$ be any relation satisfying \eqref{posCondn}, and define $U = \llbracket a \in A \mid R(a,a) \rrbracket$. Then $R$ restricted to $U$ is reflexive and symmetric, hence equal to $S^\dag \circ S$ for some relation $S$. Then $R$ is equal to $S'^\dag S'$ where $S' = \llbracket (a,u) \in A \times U \mid (a=U(u)) \wedge S(a) \rrbracket$.
  \\ For (b $\Rightarrow$ c): the Choi matrix $S$ of a relation $R$ is given by $S((a,b),(a',b')) \iff R(a'^{-1} \circ a,b' \circ b^{-1})$, which can be seen to satsify \eqref{posCondn} iff $R$ itself respects inverses.
  \\ To see (c $\Rightarrow$ a), observe that a reflexive symmetric relation $R \rightarrowtail A \times A$ defines a state $R'$ in $\CP(\RelC)$ of the special dagger Frobenius structure given by the indiscrete groupoid on $A$ (Example~\ref{ex-indiscgpd}). Positivity of the Choi matrix of $R'$ in this case means that $R$ is itself positive.
\end{proof}

A category satisfying the properties of the previous proposition is called \emph{positively regular}. 
We can now generalize the known description of $\CP(\Rel)$ of~\cite[Proposition 7.3]{Coecke2013i}.

\begin{corollary} \label{CP[PosReg]} 
  Let $\catC$ be a positively regular category. Then $\CP(\RelC)$ is equivalent to the category of groupoids and relations in $\catC$ that respect inverses.
  \qed
\end{corollary}

\begin{example} 
  $\cat{Set}$ is positively regular, since any relation $R$ satisfying~\eqref{posCondn} is equal to $S^\dag \circ S$ with $S = \left\{(a,a,b), (a,b,a) \mid R(a,b) \right\}$. More generally, any coherent category, such as $\cat{KHaus}$, is positively regular and hence so is any topos. Section~\ref{sec:malcev} will show that $\cat{Gp}$ and $\cat{Vect}_k$ are positively regular.

  The category of semigroups satisfying $(\forall x,y)\, xyx = yxy$ is regular but not positively regular; positive relations satisfy $R(a,b) \wedge R(c,d) \Rightarrow R(acb,dbc)$, but relations satisfying~\eqref{posCondn} need not. 
\end{example}

\section{Mal'cev categories} \label{sec:malcev}

This section studies a broad class of regular categories for which the CP construction takes a very natural form, allowing us to bypass the unusual notion of a relation respecting inverses.

\begin{definition} 
  A category $\catC$ with finite limits is a \emph{Mal'cev category} if every reflexive relation $R \rightarrowtail A \times A$ in $\catC$ is an equivalence relation.
\end{definition}

\begin{proposition} \label{groupodclosedequalsMal'cev} 
  The following are equivalent for a regular category:
  \begin{enumerate}[(a)]\itemsep0pt
	\item inverse-respecting subobjects $R \rightarrowtail A_1$ of internal groupoids are closed under composition;
    \item every relation $R \rightarrowtail A \times A$ is \emph{difunctional}: 
      $(\forall a,b,c,d \in A)\, R(a,b) \wedge R(c,b) \wedge R(c,d) \Rightarrow R(a,d)$;
	\item the category is Mal'cev. 
  \end{enumerate}
\end{proposition}
\begin{proof}
  For (a $\Rightarrow$ b), let $R \rightarrowtail A \times A$ be a relation. 
  Define a new relation $S \rightarrowtail R \times R$ by setting $S((a,b),(c,d))$ $\Leftrightarrow$ $R(a,d) \wedge R(c,b)$. 
  Then $S$ is reflexive, because it is defined on $R$, and symmetric. Thus $S$ defines an inverse-respecting subobject of the indiscrete groupoid on $R$, and hence is closed under composition in $R \times R$. That is, $S$ is transitive. Now suppose $R(a,b)$, $R(c,b)$, and $R(c,d)$. Then $S((a,b),(c,b))$ and $S((c,b),(c,d))$, hence $S((a,b),(c,d))$, and so $R(a,d)$.

  To see (b $\Rightarrow$ a), consider an inverse-respecting subobject $R \rightarrowtail A_1$, and set $S = \sem{(a,b) \in A_1 \times A_1 \mid (\text{dom}(a)= \text{cod}(b)) \wedge R(a,b) }$. For composable $a$ and $b$ in $A_1$, let $x = \text{dom}(a) = \text{cod}(b)$. Since $R$ is closed under $\id{\text{dom}(-)}$ and $\id{\text{cod}(-)}$, we have $S(\id{x},b)$, $S(a, \id{x})$, and $S(\id{x},\id{x})$. Now $S(a,b)$ by difunctionality, that is, $R(a \circ b)$, so $R$ is closed under composition. 

  Finally, (b $\Leftrightarrow$ c) is well-known~\cite[Proposition~1.2]{A.CarboniM.C.Pedicchio}.
\end{proof}

\begin{example} 
  Mal'cev categories were first studied in the context of universal algebra. An algebraic variety is Mal'cev precisely when it contains an operation $p(x,y,z)$ with $p(x,y,y) = p(y,y,x) = x$. Hence the categories $\cat{Gp}$ and $\cat{Vect}_k$ are regular Mal'cev, as are the categories of Lie algebras, abelian groups, rings, commutative rings, associative algebras, quasi-groups and Heyting algebras. Any abelian category such as $\cat{Mod}_R$ is Mal'cev. The opposite category of any topos is regular Mal'cev.
\end{example}

In a Mal'cev category $\catC$, the forgetful functor $\cat{Gpd}(\catC) \to \cat{Cat}(\catC)$ is an isomorphism, that is, every internal category in $\catC$ uniquely defines an internal groupoid~\cite[Theorem 2.2]{A.CarboniM.C.Pedicchio}. Moreover, $\cat{Gpd}(\catC)$ is regular when $\cat{C}$ is Mal'cev~\cite[Proposition~3.1 and Theorem~3.2]{Gran1999}. In this case $\Rel(\cat{Gpd}(\catC))$, which we now describe, makes sense. By contrast, $\Rel(\cat{Gpd}(\cat{Set}))$ is ill-defined. 

\begin{lemma}\label{subobinGpdC}
  Let $\catC$ be a regular category. Subobjects of $(A_0, A_1)$ in $\cat{Gpd}(\catC)$ are subobjects $R \rightarrowtail A_1$ in $\catC$ that are closed under $\id{\text{dom}(-)}$, inverses, and composition. 
\end{lemma}
\begin{proof}
  Morphisms $f = (f_0,f_1) \colon (C_0, C_1) \to (D_0,D_1)$ in $\cat{Gpd}(\catC)$ are determined by morphisms $f_1 \colon C_1 \to D_1$ respecting the groupoid structure, since $f_0$ can be reconstructed as $c \mapsto \text{dom}(f_1(\id{c}))$ for $c \in C_0$. Now $K_1 = \sem{ (x,y) \in C_1 \times C_1 \mid \  f_1(x) = f_1(y)}$ determines a subgroupoid $K \leq C \times C$, with $f \circ \pi_1 = f \circ \pi_2 \colon K \to D$. If $f$ is monic in $\cat{Gpd}(\catC)$, then $\pi_1 = \pi_2$ and so $f_1$ is monic in $\catC$. 
\end{proof}

\begin{theorem} \label{Thm-CPMal'cev} 
  Regular Mal'cev categories are positively regular. 
  Conversely, a positively regular category $\catC$ is Mal'cev if and only if Corollary~\ref{CP[PosReg]} provides equivalences
  \[ 
    \CP(\RelC) \simeq \Rel(\cat{Gpd}(\catC)) \simeq \Rel(\cat{Cat}(\catC)).
  \] 
\end{theorem}
\noindent Note that $\Rel(\cat{Gpd}(\catC))$ and $\Rel(\cat{Cat}(\catC))$ are ill-defined for general positively regular $\catC$.
\begin{proof}
  For the first statement, let $R$ be any reflexive symmetric relation in a regular Mal'cev category. Then $R$ is an equivalence relation and hence satisfies $R = R^\dagger \circ R$, making it positive.

 % is positive, since one can easily show that $R = R^\dagger \circ R$.

 % $R = \id{} \circ R \leq R \circ R = R^\dagger \circ R  \leq R$.

 % and so satisfies, making it positive.
%Any equivalence relation satisfies $R = R^{\dagger} \circ R$ and hence is positive, since $R = \id{} \circ R \leq R \circ R \leq R$.
  %Then $R$ is an equivalence relation and hence is positive, as noted in Lemma~\ref{lem-eqrelpos}. 

  Now let $\catC$ be a positively regular category. By Corollary~\ref{CP[PosReg]}, $\CP(\RelC)$ is equivalent to the category of groupoids and relations that respect inverses. By Lemma~\ref{subobinGpdC} and Proposition~\ref{groupodclosedequalsMal'cev}, the forgetful functor from $\Rel(\cat{Gpd}(\catC))$ to this category is (well-defined and) an isomorphism if and only if $\catC$ is Mal'cev, in which case $\cat{Gpd}(\catC)$ is isomorphic to $\cat{Cat}(\catC)$.
\end{proof}

\begin{example}
  We may read the previous theorem as saying that the CP construction, usually taken to add mixed states and processes to pure quantum theory, can be regarded as a process of \emph{categorification} in a broad class of categories of relations.
  The category $\cat{Cat}(\cat{Gp})$ of strict 2-groups is equivalent to the category $\cat{CrMod}$ of crossed modules, whence $\CP(\Rel(\cat{Gp})) \simeq \Rel(\cat{CrMod})$.
  Similarly, $\CP(\Rel(\cat{Vect}_k))$ is equivalent to the category of relations in 2-vector spaces.
\end{example}

Next we observe that, in any Mal'cev regular category, just as every internal category is in fact a groupoid, there is also redundancy in the definition of a dagger Frobenius structure.

%Next we observe that, in any Mal'cev regular category,

\begin{theorem}\label{prop:frobeniusredundant}
  If $\catC$ is a regular Mal'cev category, any pair of morphisms $\tinymult \colon A \otimes A \to A$ , $\tinyunit \colon I \to A $ in $\RelC$ satisfying unitality form the composition and identities of an internal category in $\catC$, or equivalently a dagger special Frobenius structure in $\RelC$.
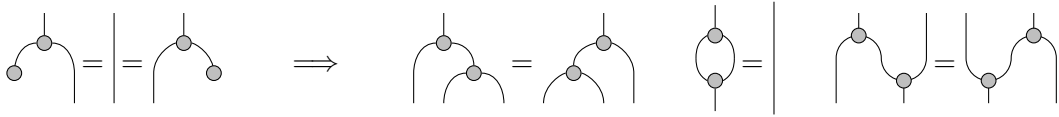
  \[ 
    \begin{pic}[scale=.4] 
      \node[dot] (d) {};
  	  \draw (d) to +(0,1);
      \draw (d) to[out=0,in=90] +(1,-1) to +(0,-1);
      \draw (d) to[out=180,in=90] +(-1,-1) node[dot] {};
    \end{pic}
    =
    \begin{pic}[scale=.4]
      \draw (0,0) to (0,3);
    \end{pic}
    =
    \begin{pic}[yscale=.4,xscale=-.4]
      \node[dot] (d) {};
      \draw (d) to +(0,1);
      \draw (d) to[out=0,in=90] +(1,-1) to +(0,-1);
      \draw (d) to[out=180,in=90] +(-1,-1) node[dot] {};
    \end{pic}
  \qquad
  \implies
  \qquad
\begin{pic}[scale=.4]
      \node[dot] (t) at (0,1) {};
      \node[dot] (b) at (1,0) {};
      \draw (t) to +(0,1);
      \draw (t) to[out=0,in=90] (b);
      \draw (t) to[out=180,in=90] (-1,0) to (-1,-1);
      \draw (b) to[out=180,in=90] (0,-1);
      \draw (b) to[out=0,in=90] (2,-1);
    \end{pic}
    =
    \begin{pic}[yscale=.4,xscale=-.4]
      \node[dot] (t) at (0,1) {};
      \node[dot] (b) at (1,0) {};
      \draw (t) to +(0,1);
      \draw (t) to[out=0,in=90] (b);
      \draw (t) to[out=180,in=90] (-1,0) to (-1,-1);
      \draw (b) to[out=180,in=90] (0,-1);
      \draw (b) to[out=0,in=90] (2,-1);
    \end{pic}  
    \qquad  
    \begin{pic}
      \node[dot] (t) at (0,.6) {};
      \node[dot] (b) at (0,0) {};
      \draw (b.south) to +(0,-0.3);
      \draw (t.north) to +(0,+0.3);
      \draw (b.east) to[out=30,in=-90] +(.15,.3);
      \draw (t.east) to[out=-30,in=90] +(.15,-.3);
      \draw (b.west) to[out=-210,in=-90] +(-.15,.3);
      \draw (t.west) to[out=210,in=90] +(-.15,-.3);
    \end{pic}
    =
    \begin{pic}
    	\draw(0,0) to (0,1.5);
    \end{pic}
  \qquad
    \begin{pic}[scale =.6]
      \draw (0,0) to (0,1) to[out=90,in=180] (.5,1.5) to (.5,2);
      \draw (.5,1.5) to[out=0,in=90] (1,1) to[out=-90,in=180] (1.5,.5) to (1.5,0);
      \draw (1.5,.5) to[out=0,in=-90] (2,1) to (2,2);
      \node[dot] at (.5,1.5) {};
      \node[dot] at (1.5,.5) {};
    \end{pic}
    =
    \begin{pic}[scale = 0.6, xscale=-1]
      \draw (0,0) to (0,1) to[out=90,in=180] (.5,1.5) to (.5,2);
      \draw (.5,1.5) to[out=0,in=90] (1,1) to[out=-90,in=180] (1.5,.5) to (1.5,0);
      \draw (1.5,.5) to[out=0,in=-90] (2,1) to (2,2);
      \node[dot] at (.5,1.5) {};
      \node[dot] at (1.5,.5) {};
    \end{pic}
  \]
\begin{proof}
  Since any internal category in $\catC$ is a groupoid, the second statement follows from Theorem~\ref{Thm-OperatorAreGroupoids}. 
Let $\tinymult = (M \colon A \times A \relto A)$ and $\tinyunit = (U \colon I \relto A)$. Then unitality corresponds to the formulae
\begin{align}
  (\exists x \in U) \  M(x,a,a') & \iff a = a' \label{units2_monproof} \tag{U1} & \\ 
  (\exists x \in U) \  M(a,x,a') & \iff a = a' \label{units1_monproof} \tag{U2} &
\end{align}

 % It follows from~\eqref{special_monproof} that $M$ is single-valued as a relation, as we saw in the proof of Proposition~\ref{prop:frobeniusredundant}. 
  
 We will first show that $M$ is single-valued as a relation. Suppose that both $M(a,b,c)$ and $M(a,b,d)$, then by \eqref{units2_monproof} $(\exists x \in U) \ M(c,x,c)$, and applying difunctionality we have $M(c,x,d)$ and so $c=d$. 
 In any regular category, such a relation is represented by a subobject of the form $(A \times A \leftarrowtail B \sxto{m} A)$ in $\catC$, where $B = \sem{ (a ,b) \in (A \times A) \mid  (\exists c \in A) \ M(a,b,c)}$; see \cite[Lemma~2.8]{Butz1998}. Write $(a,b)\defined$ for $B(a,b)$. Define the following relations in $\catC$:
  \begin{align}
    S & = \sem{ (a,x) \in A \times U \mid  (a,x)\defined } \colon A \relto U  \label{SourceRelation} \\
    T & = \sem{ (a,y) \in A \times U \mid  (y,a)\defined } \colon A \relto U \label{TargetRelation}
  \end{align}
  just as in~\cite[Definition~2]{Heunen2013}. Now $S$ is total, by \eqref{units1_monproof}, and single-valued since if $M(a,x,a)$ and $M(a,y,a)$ for $x$, $y$ in $U$, then since unitality gives $M(x,x,x)$ we have $M(x,y,x)$ by difunctionality, and so $x=y$. The same holds for $T$, and so $S$, $T$ correspond to morphisms $s$, $t$ in $\catC$ defining the data of an internal category
    \[\begin{pic}[font=\small,xscale=2.5]
    \node (0) at (-1,0) {$U$};
    \node (1) at (0,0) {$A$};
    \node (11) at (1,0) {$A \times_U A = B$};
    \draw[->] (11) to node[above=-1mm] {$m$} (1);
    \draw[>->] (0) to (1);
    \draw[->] ([yshift=-2mm]1.west) to node[below=-1mm] {$s$} ([yshift=-2mm]0.east);
    \draw[->] ([yshift=2mm]1.west) to node[above=-1mm] {$t$} ([yshift=2mm]0.east);
  \end{pic}\]
  where it remains to show that $B$ is in fact a pullback of $s$ and $t$, \textit{i.e.} $(a,b)\defined$ if and only if $s(a) = t(b)$. Suppose that $(a,b)\defined$ and $s(a) = x$. Then we have $(a,b)\defined$, $(a, x)\defined$, $(x,x)\defined$ and hence $(x,b)\defined$ by difunctionality. Conversely, $(x,b)\defined$, $(a,x)\defined$,  and $(x,x)\defined$ implies that $(a,b)\downarrow$.
Finally, in \cite[Theorem~2.2]{A.CarboniM.C.Pedicchio} it is shown that, in any Mal'cev category, all of the equations required of an internal category follow automatically whenever $m(a,s(a)) = a$, $m(t(a),a) = a$, and $s(x) = x = t(x)$ for all $a \in A$ and $x \in U$. 
\end{proof}
\end{theorem}
Note that unitality is a property of $\tinymult$ alone, since we must have $ \tinyunit = \sem{a \in A \mid \tinymult(a,a,a)}$ .

\begin{example} Crossed modules are known to have numerous equivalent descriptions, as groups in $\cat{Cat}$ or $\cat{Gpd}$, and as categories or groupoids in $\cat{Gp}$ \cite[3.1]{Noohi2007} \cite{Baez2004a}. We can now add that a crossed module is just a unital morphism $\tinymult$ in $\Rel(\cat{Gp})$. Similarly, a  2-vector space is simply a unital morphism in $\Rel(\cat{Vect_K})$. 
\end{example}
\section{Quantum properties of $\RelC$}\label{sec:properties}

The category $\Rel$ of sets and relations is compact dagger, like $\cat{FHilb}$, but fails to satisfy many properties that are seen as typical to quantum theory. This section presents three such properties of $\cat{FHilb}$ that fail in $\Rel$ but hold in $\RelC$ whenever $\catC$ is a regular Mal'cev category. In this sense $\RelC$ is a `more quantum' model than $\Rel$ for regular Mal'cev categories $\catC$. 
%We conclude by countering this with some non-quantum aspects of $\RelC$ which are present when $\catC$ is $\cat{Gp}$ or $\cat{Vect}_k$ but not $\cat{Set}$.
We conclude by countering this with some non-quantum features of our main $\RelC$ examples not shared by $\Rel$.

%in our main examples, but not in $\Rel$

%We conclude by countering this with some non-quantum features present in our main $\RelC$ examples, but not in $\Rel$.

\paragraph{Heisenberg uncertainty}

The Heisenberg uncertainty principle states that no information can be obtained from a quantum system without disturbing its state~\cite[Section~5.2]{heinosaariziman:quantum}. To model it categorically, we need an appropriate notion of quantum system. Following~\cite{Coecke2013i}, we will take \emph{quantum structures} to be special dagger Frobenius structure of the form $B=(A^* \otimes A, d
\begin{pic}[dotpic,scale=0.33]
% \node[circle, draw=black, scale=0.7] at (-7.5,0){\large $d$};
		\node [style=none] (0) at (-1.25, -1.25) {};
		\node [style=none] (1) at (-5.75, -1.25) {};
		\node [style=none] (2) at (-2.75, 1.25) {};
		\node [style=none] (3) at (-4.25, 1.25) {};

		\node [style=none] (6) at (-4.25, -1.25) {};
		\node [style=none] (7) at (-2.75, -1.25) {};

		\draw [in=-90, out=90] (0.center) to (2.center);
		\draw [in=90, out=-90] (3.center) to (1.center);
		\draw [in=90, out=90, looseness=2.00] (6.center) to (7.center);
\end{pic}
,  d^{-1}
\begin{pic}[dotpic,scale=0.33]
% \node[circle, draw=black, scale=0.4] at (2,0){\huge $d^{-1}$};
		\node [style=none] (4) at (3.75, 0.25) {};
		\node [style=none] (5) at (5.25, 0.25) {};
		\draw [in=-90, out=-90, looseness=2.00] (4.center) to (5.center);	
\end{pic}
)$, where the scalar $d\colon I \to I$ is invertible. 
In $\cat{FHilb}$, these correspond to the C*-algebras $\mathbb{M}_n$ of $n$-by-$n$ matrices. 
In $\RelC$ for regular $\catC$, these correspond to indiscrete groupoids on $A$ (see Example~\ref{ex-indiscgpd}).
We model the principle contrapositively as follows, abstracting the fact that POVMs on a Hilbert space $A$ with $n$ outcomes are precisely morphisms $B \to B \otimes C$ in $\CPs[\cat{C}]$, where $C$ is a commutative special dagger Frobenius structure on $\mathbb{C}^n$. 

\begin{definition}
  A compact dagger category satisfies the \emph{Heisenberg uncertainty principle} when the following holds for any quantum structure $(B,\tinymult[white_dot],\tinyunit[white_dot])$, any commutative special dagger Frobenius structure $(C,\tinymult,\tinyunit)$, and any completely positive morphism $M \colon B \to B \otimes C$:
  \begin{equation}\label{heisenberg}
    \begin{pic}
      \node[morphism](M) at (0,0){$M$};
	  \draw([xshift=1mm]M.north east) to +(0,0.4) node[above,dot](d){};
	  \draw([xshift=-1mm]M.north west) to +(0,0.5) node[above]{$B$};
	  \draw(M.south) to (0,-0.5) node[below]{$B$};
	\end{pic}
    \;= 
	\begin{pic}
	  \draw (0,-0.5) node[below] {$B$} to (0,0.73) node[above] {$B$};
	\end{pic}
	\qquad	\Longrightarrow \qquad
	(\exists \psi \colon I \to C)\quad
	\begin{pic}
	  \node[morphism](M) at (0,0){$M$};
	  \draw([xshift=-1mm]M.north west) to +(0,0.4) node[above,white_dot](d){};
	  \draw([xshift=1mm]M.north east) to +(0,0.5) node[above]{$C$};
	  \draw(M.south) to (0,-0.69) node[below]{$B$};
	\end{pic}
	\;=
	\begin{pic}
	  \node[white_dot](d) at (0,1){};
	  \node[morphism](s) at (0,1.5){$\psi$};
	  \draw (s.north) to (0,2) node[above]{$C$};
	  \draw (d.south) to +(0,-0.3) node[below] {$B$};
	\end{pic}
  \end{equation}
\end{definition}

The category $\FHilb$ satisfies the Heisenberg uncertainty principle~\cite[Section~6.3]{maassen:probability}.

\begin{lemma}
  The category $\Rel$ does not satisfy the Heisenberg uncertainty principle.
\end{lemma}
\begin{proof}
 (See also \cite{msc}) Let $B$ be the indiscrete groupoid on the two-element set $\left\{{x,y}\right\}$, and let $C$ be the group $\mathbb{Z}_2=\{0,1\}$, regarded as a groupoid. The inverse-respecting relation 
  \[
    M = \{ (b,b,0))\mid b \in B \} \cup \{((x,x),(x,x),1))\} \subseteq B \times (B \times C).
  \]
  satisfies $M((x,x),(x,x),1)$ but not $M((y,y),(y,y),1)$.

\end{proof}

\begin{proposition} 
  Let $\catC$ be a regular Mal'cev category. Then $\RelC$ satisfies the Heisenberg uncertainty principle for any quantum structure $(B, \tinymult[white_dot], \tinyunit[white_dot])$ and \emph{any} special dagger Frobenius structure $(C, \tinymult, \tinyunit)$.
\end{proposition}
\begin{proof}
  Regard $M$ as a relation $M \rightarrowtail (A \times A) \times (A \times A) \times C$, where $B$ is the indiscrete groupoid on $A$. Suppose $M$ satisfies the left-hand side of~\eqref{heisenberg}, so that
  \begin{equation} \label{no-disturbance}
    (\exists x \in C)\; M((a,a'), (b, b'), \id{\text{dom}(x)}) \;\iff\; (a, a') = (b, b').
  \end{equation}
  We wish to show $(\exists b \in A) \ M( (a,a'), (b,b), x) \iff a=a' \wedge \psi(x)$ for some $\psi \rightarrowtail C$. Define the new relation $T = \sem{ (a,x) \in A \times C \mid M((a,a),(a,a),x)}$. Then by~\eqref{no-disturbance} and closure of $M$ under $\id{\text{dom}}$ and $\id{\text{cod}}$, we have $(\exists b \in A) \ M( (a,a'), (b,b), x)$ if and only if $(a=a') \wedge T(a,x)$. It suffices to show that $T(a,x)$ holds precisely when $(\exists a' \in A) \  T(a',x)$, because we may then take $\psi = \sem{ x \in C \mid (\exists a\in A) \ T(a,x)}$.

  Once more using~\eqref{no-disturbance} and closure of $M$ under $\id{\text{dom}}$ and $\id{\text{cod}}$, observe that $(\forall a, a'\in A) \  (\exists x \in C) $ such that $T(a,x)\wedge T(a',x)$. Hence if $T(a',x)$ holds then $(\exists y \in C) \ T(a,y) \wedge T(a',y) \wedge T(a',x)$ and so by difunctionality $T(a,x)$ holds, as desired.  
\end{proof}

\paragraph{Broadcasting}

While statistical mechanics includes its own version of the no-cloning theorem, it has instead been argued that one of the unique features of classical systems is in their capacity to be \emph{broadcast}~\cite{Barnum2007}. We now capture this property categorically, following~\cite{Coecke2013i}.

\begin{definition} 
  Let $\cat{C}$ be a compact dagger category. A \emph{broadcasting map} for an object $(A, \tinymult)$ of $\CP(\catC)$ is a morphism $A \sxto{B} A \otimes A$ in $\CP(\catC)$ satisfying:
  \[ \label{broadcasting}
    \begin{pic}
      \node[morphism](M) at (0,0){$B$};
	  \draw([xshift=1mm]M.north east) to + (0,0.3) node[above,dot]{}; 
	  \draw([xshift=-1mm]M.north west) to +(0,0.5) node[above]{$A$};
	  \draw(M.south) to (0,-0.5) node[below]{$A$};
	\end{pic}
	\quad = \;
	\begin{pic}
	  \draw (0,0) node[below]{$A$} to (0,1.29) node[above]{$A$};
	\end{pic}
	\; = \quad
	\begin{pic}
	  \node[morphism](M) at (0,0){$B$};
	  \draw([xshift=-1mm]M.north west) to + (0,0.3) node[above,dot]{}; 
	  \draw([xshift=1mm]M.north east) to +(0,0.5) node[above]{$A$};
	  \draw(M.south) to (0,-0.5) node[below]{$A$};
	\end{pic}
  \]
  Any commutative dagger Frobenius structure in $\catC$ has a broadcasting map $\tinymultflip$, which can be shown to be completely positive. We say $\catC$ satisfies the \emph{no-broadcasting principle} if the converse holds.
\end{definition}

The category $\cat{FHilb}$ satisfies the no-broadcasting principle~\cite{Barnum1996}.

\begin{lemma}\cite{heunenkissinger:cbh}
  The category $\Rel$ does not satisfy the no-broadcasting principle.
\end{lemma}
\begin{proof}
  Let $G$ be a nonabelian group, and regard it as a groupoid $(G_0,G_1)$.
  Define a morphism $B \colon G_1 \times G_1 \to G_1$ in $\Rel$ by 
  $
    B = \left\{ (g,\id{\text{dom}(g)}, g) \mid g \in G_1 \right\} \cup \left\{ (g, g, \id{\text{dom}(g)}) \mid g \in G_1 \right\}
  $.
  This relation respects inverses, so is a morphism of $\CP(\Rel)$. It is also a broadcasting map since $(\exists x \in G_0) \ B(g,\id{x},h) \iff (\exists x \in G_0) \ B(g,h,\id{x}) \iff g=h$ for $g,h \in G_1$. 
\end{proof}

\begin{proposition} 
  For $\cat{C}$ regular Mal'cev, $\RelC$ satisfies the no-broadcasting principle.
\begin{proof}
  Suppose $B \rightarrowtail A \times A \times A$ in $\CP(\RelC)$ is broadcasting. Then:
  \begin{equation} \label{broadcasting-eq}
    (\exists x \in A) \ B(a, \id{\text{dom}(x)}, a') \iff a = a' \iff (\exists x \in A) \ B(a, a', \id{\text{dom}(x)})
  \end{equation}
  Use closure of $B$ under identities and inverses to show that, in any regular category $\catC$, \eqref{broadcasting-eq} implies ${(\forall a \in A)}\ \text{dom}(a) = \text{cod}(a)$. Now, arguing in the internal logic of $\catC$,  let $a,a' \in A$ be such that $\text{dom}(a) = \text{dom}(a')$, so that $B(a,\id{\text{dom}(a)},a)$ and $B(a',a',\id{\text{dom(a)}})$. Since $\catC$ is Mal'cev, $B$ is closed under composition in $A$ by Proposition~\ref{groupodclosedequalsMal'cev}, and so $B(a \circ a' \circ a^{-1}, a', \id{\text{dom}(a)})$. Hence $a \circ a' = a' \circ a$ by~\eqref{broadcasting-eq}. 
\end{proof}
\end{proposition}

\paragraph{Rank}
 
The third property we discuss concerns the linear structure of quantum theory. Due to this structure, morphisms in $\cat{FHilb}$ come with a notion of rank. Rays, morphisms of rank at most one, are reflected in the graphical calculus by disconnectedness. 

\begin{definition} \label{def:rays}
  We say a monoidal dagger category satisfies the \emph{bottleneck principle} if morphisms $R$ factor through $I$ whenever $R^\dag \circ R$ does so:
  \[
	(\exists \psi, \phi \colon I \to A)\quad
    \begin{pic}
      \node[morphism,hflip] (g) at (0,1.75) {$R\vphantom{f}$};
      \node[morphism] (f) at (0,1) {$R$};
      \draw (f.south) to +(0,-.3) node[below] {$A$};
      \draw(f.north) to (g.south);
      \draw (g.north) to +(0,.3) node[above] {$A$};
    \end{pic} 
    \quad = \quad
    \begin{pic}
      \node[morphism,hflip](b) at (0, 1){$\phi$};
      \node[morphism](t) at (0, 1.75){$\psi$};
      \draw (b.south) to +(0,-0.3) node[below] {$A$};
      \draw (t.north) to +(0,0.3) node[above] {$A$};
    \end{pic}
    \quad \implies \quad
	(\exists \phi \colon I \to A, \psi \colon I \to B)\quad
    \begin{pic}
      \node[morphism] (r) at (0,1) {$R$};
      \draw (r.north) to +(0,0.66) node[above]{$B$};
      \draw (r.south) to +(0,-0.66) node[below]{$A$};
    \end{pic} 
    \quad = \quad
    \begin{pic}
	  \node[morphism,hflip](b) at (0, 1){$\phi$};
	  \node[morphism](t) at (0, 1.75){$\psi$};
	  \draw (b.south) to +(0,-0.3) node[below] {$A$};
	  \draw (t.north) to +(0,0.3) node[above] {$B$};
    \end{pic}
  \]
\end{definition}

The category $\cat{FHilb}$ clearly satisfies the bottleneck principle, because $\mathrm{rank}(f^\dag \circ f) = \mathrm{rank}(f)$.
\begin{lemma} 
 The category $\cat{Rel}$ does not satisfy the bottleneck principle.
\begin{proof}
  Let $B = \left\{ {0,1}\right\}$, and consider the relation $R \colon B \relto B$ given by $R = \left\{ (0,0), (0,1), (1,1) \right\}$. Now $R^\dag \circ R$ splits as the product relation $B \times B$, but $R$ cannot be written as a product of subsets of $B$.
\end{proof}
\end{lemma}
The previous lemma leads to unusual behaviour from a quantum perspective in $\Rel$. For example, taking the partial trace of an entangled state $\Psi$ can result in a pure state:
\[
  \begin{pic}
    \node[morphism, vflip] (l) at (0,0) {$\Psi$}; 
	\node[morphism](r) at (1,0) {$\Psi$}; 
	\draw ([xshift = -1mm]r.north west) to[out=90, in=90] ([xshift=1mm]l.north east);
	\draw ([xshift=1mm]r.north east) to +(0,0.5);
	\draw ([xshift=-1mm]l.north west) to +(0,0.5);
  \end{pic}
  \quad = \quad
  \begin{pic}
	\node[morphism, vflip] (l) at (0,0) {$\psi$}; 
	\node[morphism](r) at (1,0) {$\psi$}; 
	\draw (r.north) to +(0,0.5);
	\draw (l.north) to +(0,0.5);
  \end{pic}
\]
which cannot occur for entangled states in $\cat{FHilb}$. 
\begin{lemma}
  The category $\RelC$ satisfies the bottleneck principle for regular Mal'cev categories $\catC$.
\begin{proof}
 A relation $R \colon A \relto B$ disconnecting in the above sense means $R(a,b) \wedge R(a',b') \Rightarrow R(a,b')$, in which case $R(a,b) \iff (\exists a' \in A) R(a',b) \wedge (\exists b'\in B) R(a,b')$. Suppose $R^\dag \circ R$ splits as in the left-hand side of Definition~\ref{def:rays}, and assume $R(a,b)$ and $R(a',b')$; we will show that $R(a,b')$. It follows from $(R^\dag \circ R)(a,a)$ and $(R^\dag \circ R)(a',a')$ that $(R^\dag \circ R)(a,a')$, that is, $(\exists e \in A) \ R(a,e) \wedge R(a',e)$. But then $R(a,e)$, $R(a',e)$ and $R(a',b')$, and so $R(a,b')$ holds by difunctionality. 
\end{proof}
\end{lemma}

\paragraph{Projections}

To finish, we show that the models $\RelC$ for regular Mal'cev $\catC$, despite the above results, have some non-quantum features distinguishing them from $\cat{FHilb}$. For example, by Proposition~\ref{groupodclosedequalsMal'cev}, any state $\psi \colon I \to (A, \tinymult)$ in $\CP(\RelC)$ is a \emph{projection}~\cite[Definition~3.1]{Coecke2013}:
\[
\begin{pic}
\node[morphism] (l) at (0,0) {$\psi$}; 
\node[morphism](r) at (1,0) {$\psi$};
\node[dot] (d) at (0.5,0.5){}; 
\draw (d.west) to[out=180, in=90] (l.north);
\draw (d.east) to[out=0, in=90] (r.north);
\draw (d.north) to +(0,0.5) node[above]{$A$};
\end{pic}
\quad
=
\quad
\begin{pic}
\node[morphism](r) at (0,0){$\psi$};
\draw (r.north) to +(0,0.9) node[above]{$A$};
\end{pic}
\quad
=
\quad
\begin{pic}
\node[dot](d1) at (0,0){};
\node[dot](d2) at (0,0.5){};
\node[morphism, hflip](r) at (-0.5,1){$\psi$};

\draw (r.south) to[out = -90, in = 180] (d2.west);
\draw (d2.east) to[out = 0, in = -90] +(0.5,0.5) to +(0,0.2) node[above]{$A$};
\draw (d1.north) to (d2.south);

\end{pic}
\]
This is not the case in $\cat{FHilb}$, where, up to scalar factors, the projections of a quantum structure are precisely projections in the usual sense, while states are arbitrary density matrices. \\

\paragraph{Unique measurement outcomes}

Another more striking difference from $\cat{FHilb}$ is that many of the categories $\RelC$ lack distinct classical outcomes of experiments. These outcomes are represented categorically by states copied by the map $\tinymultflip$ of a commutative Frobenius structure \cite{coecke2011interacting}.

We call an object $A$ of a regular category $\catC$ \emph{inhabited} if $\sem{ \exists a \in A} = \id{1}: 1 \rightarrowtail 1$, and $\catC$ \emph{entirely inhabited} when this holds for all objects. Equivalently, any subobject of a terminal object $1$ is isomorphic to $1$.

\begin{proposition} Let $\catC$ be an entirely inhabited regular category. Then any two copyable states $H$, $T$ of a special dagger Frobenius structure $(A, \tinymult)$ in $\RelC$ are equal.

\[ 
    \begin{pic}[scale=.35, yscale=-1] 
      \node[dot] (d) {};
      \draw (d) to +(0,1.6) node[morphism]{$H$};
      \draw (d) to[out=0,in=90] +(1,-1) to +(0,-0.25);
      \draw (d) to[out=180,in=90] +(-1,-1) to +(0,-0.25);
    \end{pic}
    =
    \begin{pic}[scale=.35, yscale = -1]
      \draw (0,0) to +(0,3) node[morphism]{$H$};
      \draw (2.5,0) to +(0,3) node[morphism]{$H$};
    \end{pic}
    \qquad
     \begin{pic}[scale=.35, yscale=-1] 
      \node[dot] (d) {};
      \draw (d) to +(0,1.6) node[morphism]{$T$};
      \draw (d) to[out=0,in=90] +(1,-1) to +(0,-0.25);
      \draw (d) to[out=180,in=90] +(-1,-1) to +(0,-0.25);
    \end{pic}
    =
    \begin{pic}[scale=.35, yscale = -1]
      \draw (0,0) to +(0,3) node[morphism]{$T$};
      \draw (2.5,0) to +(0,3) node[morphism]{$T$};
    \end{pic}
    \qquad
    \implies
    \qquad
        \begin{pic}[scale=.5]
    \node[morphism](s) at (0,0) {$H$} ;
      \draw (s.north) to (0,2);
    \end{pic}
    =
        \begin{pic}[scale=.5]
    \node[morphism](s) at (0,0) {$T$} ;
      \draw (s.north) to (0,2);
    \end{pic}
\]

\begin{proof}
Arguing just as in $\Rel(\cat{Set})$, any copyable state $H$ of an internal groupoid $(A, \tinymult)$ is easily seen to satisfy $\text{dom}(a) = \text{cod}(a')$ for all its members $a, a'$, and $(\text{dom}(a) = \text{dom}(b))\wedge H(b) \implies H(a)$ for any $a$, $b$ in $A$.
Now by assumption $H \wedge T$ is inhabited, meaning $(\exists b \in A) \ H(b) \wedge T(b)$. Hence $T(a)$ implies that $\text{dom}(a) = \text{dom}(b)$ with $b$ in $H$, and so $H(a)$ holds also.
\end{proof}
\end{proposition}

% \begin{example} Any regular category with a zero object is entirely inhabited. This includes any abelian category, along with our main examples $\cat{Gp}$ and $\cat{Vect}_k$. The category of quasi-groups is Mal'cev regular but not entirely inhabited.
% \end{example}

% Hence the categories $\Rel(\cat{Gp})$ and $\Rel(\cat{Vect}_k)$ will be unable to model protocols such as `Bit-Commitment' \cite{} which require distinct copyable states to represent classical bits. An `ideal' choice of $\catC$ for a toy model would be Mal'cev regular, while not being entirely inhabited.

\begin{example} Any regular category with a zero object is entirely inhabited. This includes any abelian category, along with our main examples of Mal'cev categories $\cat{Gp}$ and $\cat{Vect}_k$. Hence these categories will be unable to model protocols which require distinct copyable states to represent classical bits. 
An `ideal' choice of $\catC$ for a toy model would be Mal'cev regular, while not being entirely inhabited:
the category of quasi-groups provides such an example. 
\end{example}

%Our main examples of $\Rel(\Grp)$ and $\Rel(\cat{Vec})_K$ thus cannot model these protocols. However, this does not prevent us modelling other processes such as quantum teleportation which can be represented `purely' in diagrams.

%Distinct outcomes of measurements are represented categorically by \emph{copyable states}

%corresponds to the collection of morphisms $x \to x$ for some object $x$ of the underlying groupoid, with $x$ disconnected from all other objects. In particular, any copyable state such as $H$ satisfies $\text{dom}(a) = \text{cod}(a')$ for all their members $a, a'$, and $H(a) \wedge (\text{dom}(a) = \text{dom}(b)) \implies H(b)$ for any $a$, $a'$ in $A$. 

% \begin{align}
% H(a) \implies H(a) \wedge ((\exists b \in A) H(b) \wedge T(b)) \\
% \implies \text{dom}(b) = \text{dom}(a) \wedge T(b) \\
% \implies T(a)
% \end{align}

% New section: Lack of distinct copyable states??
% Lack of biproducts? Might as well write it up!
% But lack of biproducts doesn't seem completely clear...

\bibliographystyle{eptcs}
\bibliography{regular}

\newpage
\appendix
\section{Proof of Theorem~\ref{Thm-OperatorAreGroupoids} for regular categories} \label{appendix:OpStructsAreGroupoids}

\begin{proof}
  Let $\catC$ be regular and let $(A, \tinymult, \tinyunit)$ be a special dagger Frobenius structure in $\RelC$, with $\tinymult = (M \colon A \times A \relto A)$, and $\tinyunit = (U \rightarrowtail A)$. We already saw how unitality is interpreted as the two regular formulae \eqref{units2_monproof}, and~\eqref{units1_monproof}, while the other equations of Definition~\eqref{def:frobeniusstructure} translate into:
\begin{align}
 &(\exists e \in A) \ M(a,b,e) \wedge M(e,c,d) \iff  (\exists e \in A) \  M(a,e,d) \wedge M(b,c,e) \label{assoc_monproof}  \tag{A}  \\
&(\exists b \in A)(\exists c \in A) \ M(b,c,a) \wedge M(b,c,a') \iff  a=a' \label{special_monproof} \tag{S} & \\
&(\exists e \in A) \ M(a,e,c) \wedge M(e,d,b) \iff  (\exists e \in A) \ M(c,e,a) \wedge M(e,b,d) \label{Frob'}  \tag{F} 
\end{align}   
  We follow the same proof strategy as in Theorem~\ref{prop:frobeniusredundant}. It follows from~\eqref{special_monproof} that $M$ is single-valued as a relation and hence corresponds to a subobject of the form $(A \times A \leftarrowtail B \sxto{m} A)$ in $\catC$. Again, we write $(a,b)\defined$ for $B(a,b)$, so that  $M(a,b,c)$ means that $(a,b)\defined$ and $m(a,b) = c$, and define relations $S: A \relto U$, $T: A \relto U$ by \eqref{SourceRelation} and \eqref{TargetRelation}, as well as 
  \begin{align*}
    I & = \sem{ (a,b) \in A \times A \mid  (\exists x \in U)(\exists y \in U) \   M(a,b,x) \wedge M(b,a,y) } \colon A \relto A
  \end{align*}
  %just as in~\cite[Definition~2]{Heunen2013}. 
  It suffices to show these relations are total and single-valued, as they then correspond uniquely to morphisms $s$, $t$ and $i$ in $\catC$ defining the data of an internal groupoid 
  \[
  \begin{pic}[font=\small,xscale=2.5]
    \node (0) at (-1,0) {$U$};
    \node (1) at (0,0) {$A$};
    \node (11) at (1,0) {$A \times_U A = B$};
    \draw[->] (11) to node[above=-1mm] {$m$} (1);
    \draw[>->] (0) to (1);
    \draw[->] ([yshift=-2mm]1.west) to node[below=-1mm] {$s$} ([yshift=-2mm]0.east);
    \draw[->] ([yshift=2mm]1.west) to node[above=-1mm] {$t$} ([yshift=2mm]0.east);
    \draw[->] (1.110) to[out=110,in=180] +(.05,.3) to[out=0,in=45] node[right=-1mm]{$i$} (1.70);
  \end{pic}
  \]
  where we must also show that $B$ is in fact a pullback of $s$ and $t$. 

  From the unit laws~\eqref{units2_monproof}, \eqref{units1_monproof} and associativity~\eqref{assoc_monproof}, deduce that elements of $U$ only compose when they are equal, and then that if $(a,x)\defined$ and $(a,y)\defined$ we have $(x,y)\defined$ by associativity, and so $x=y$. Hence $S$ is total and single-valued, as is $T$ similarly. The special case of~\eqref{Frob'} in which $b=s(a)$, $c=t(a)$ and $d=a$ shows that $I$ is total:
  \[
    (\exists e \in A)  \   M(a,e,s(a)) \wedge M(e,a,t(a)),
  \]
  that is, `every morphism has an inverse'. Uniqueness of inverses then follows as for any category, once we have shown that the composition $m$ is associative. 
  Writing $a^{-1}$ for any inverse of $a$, associativity~\eqref{assoc_monproof} gives $m(a^{-1},a) = s(a)$, and it follows that $a$ and $b$ are composable whenever $s(a) = t(b)$. 
  Conversely, when $a$ and $b$ are composable, $m(a,b) = m(m(a,s(a)),b) = m(a, m(s(a),b))$ by~\eqref{assoc_monproof} and so $s(a) = t(b)$. Hence $B$ is indeed the pullback of $s$ and $t$. 

  It remains to verify that these morphisms satisfy the equations defining an internal groupoid. Associativity of $m$ only requires further that $(a,b)\defined$ and $(b,c)\defined$ imply $(m(a,b),c)\defined$. From~\eqref{assoc_monproof} we find that $(m(a,b), s(b))\defined$ whenever $a$ and $b$ compose, and hence $s(m(a,b)) = s(b)$ as desired. Finally, that inverses behave as expected follows from the definition of $I$. 

  Thus any dagger special Frobenius structure in $\RelC$ defines an internal groupoid in $\catC$. Note also that this is the only possible choice of $s$, $t$ and $i$ compatible with $M$ and $U$ since any groupoid operations must satisfy the formulae defining $S$, $T$ and $I$.

  Conversely, given any internal groupoid $(C_0, C_1, m,s,t,u,i)$ in $\catC$, we must show that $M = \tinymult = (m \colon C_1 \times C_1 \relto C_1)$ and $U = \tinyunit = (u \colon C_0 \rightarrowtail C_1)$ satisfy the formulae~\eqref{special_monproof}, \eqref{assoc_monproof}, \eqref{units2_monproof}, \eqref{units1_monproof}, and~\eqref{Frob'}.
  Speciality~\eqref{special_monproof} simply states that the relation $\tinymult$ is single-valued and surjective, which holds since $m(a,s(a)) = a$ for any $a$ in $C_1$. 
  Equation~\eqref{assoc_monproof} follows from associativity of composition $m$.  
  Unitality~\eqref{units2_monproof} and~\eqref{units1_monproof} follows from the equations satisfied by $u$, $s$ and $t$. 
  Finally, the Frobenius law~\eqref{Frob'} simply amounts to the statement that $m(a^{-1},c) = m(b,d^{-1})$ if and only if $m(c^{-1},a) = m(d,b^{-1})$.  
\end{proof}

\section{The category $\CP(\cat{C})$}\label{appendix:cp}

\newenvironment{verticalhack}
  {\begin{array}[b]{@{}c@{}}\displaystyle}
  {\\\noalign{\hrule height0pt}\end{array}}
\begin{proposition}
  The category $\CP(\cat{C})$ is well-defined.
\end{proposition}
\begin{proof}
  It inherits identities and composition from $\cat{C}$. What we have to show is that composition is well-defined. Suppose that the Choi matrices of both morphisms $(A,\tinymult[white_dot],\tinyunit[white_dot]) \sxto{f} (B,\tinymult,\tinyunit) \sxto{g} (C,\tinymult[black_dot],\tinyunit[black_dot])$ are positive, say with square roots $\sqrt{f}$ and $\sqrt{g}$. Then $g \circ f$ also has a positive Choi matrix:
  
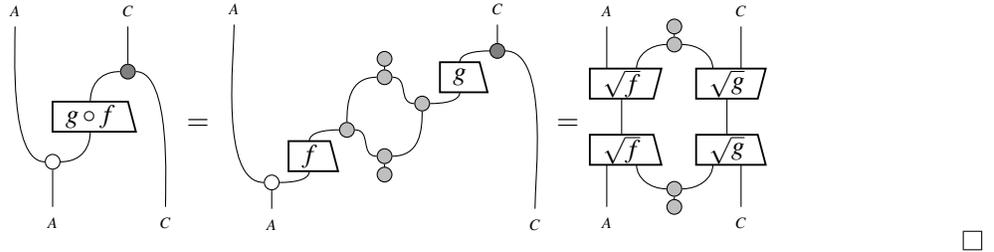
\[\begin{verticalhack}
    \begin{pic}[yscale=.6]
      \node[morphism] (f) at (0,0) {$g \circ f$};
      \node[white_dot] (l) at (-.5,-1) {};
      \node[black_dot] (r) at (.5,1) {};
      \draw (l.east) to[out=0,in=-90] (f.south);
      \draw (f.north) to[out=90,in=180] (r.west);
      \draw (r.north) to (.5,2) node[above] {$C$};
      \draw (l.south) to (-.5,-2) node[below] {$A$};
      \draw (r.east) to[out=0,in=90,looseness=.4] (1,-2) node[below] {$C$};
      \draw (l.west) to[out=180,in=-90,looseness=.4] (-1,2) node[above] {$A$};
    \end{pic}
    =
    \begin{pic}[yscale=.7]
      \node[white_dot] (l) at (-1.5,-1.25) {};
      \node[morphism] (f) at (-1,-.75) {$f$};
      \node[black_dot] (r) at (1.5,1.25) {};
      \node[morphism] (g) at (1,.75) {$g$};
      \node[dot] (a) at (-.5,-.25) {};
      \node[dot] (b) at (0,-.75) {};
      \node[dot] (c) at (0,.75) {};
      \node[dot] (d) at (.5,.25) {};
      \draw (l.south) to (-1.5,-1.75) node[below] {$A$};
      \draw (r.north) to (1.5,1.75) node[above] {$C$};
      \draw (l.west) to[out=180,in=-90,looseness=.5] (-2,1.75) node[above] {$A$};
      \draw (r.east) to[out=0,in=90,looseness=.5] (2,-1.75) node[below] {$C$};
      \draw (f.north) to[out=90,in=180] (a.west);
      \draw (a.east) to[out=0,in=180] (b.west);
      \draw (b.south) to (0,-1.1) node[dot] {};
      \draw (b.east) to[out=0,in=-90] (d.south);
      \draw (d.east) to[out=0,in=-90] (g.south);
      \draw (a.north) to[out=90,in=180] (c.west);
      \draw (c.north) to (0,1.1) node[dot] {};
      \draw (c.east) to[out=0,in=180] (d.west);
      \draw (l.east) to[out=0,in=-90] (f.south);
      \draw (g.north) to[out=90,in=180] (r.west);
    \end{pic}
    =
    \begin{pic}[yscale=.8]
      \node[morphism,hflip] (tl) at (-.7,0.55) {$\sqrt{f}$};
      \node[morphism] (bl) at (-.7,-0.55) {$\sqrt{f}$};
      \node[morphism,hflip] (tr) at (.7,0.55) {$\sqrt{g}$};
      \node[morphism] (br) at (.7,-0.55) {$\sqrt{g}$};
      \draw (bl.north) to (tl.south);
      \draw (br.north) to (tr.south);
      \draw (tl.north west) to ([yshift=7mm]tl.north west) node[above] {$A$};
      \draw (bl.south west) to ([yshift=-7mm]bl.south west) node[below] {$A$};
      \draw (tr.north east) to ([yshift=7mm]tr.north east) node[above] {$C$};
      \draw (br.south east) to ([yshift=-7mm]br.south east) node[below] {$C$};
      \node[dot] (t) at (0,1.2) {}; 
      \node[dot] (b) at (0,-1.2) {};
      \draw (tl.north east) to[out=90,in=180] (t);
      \draw (tr.north west) to[out=90,in=0] (t);
      \draw (t.north) to (0,1.5) node[dot] {};
      \draw (bl.south east) to[out=-90,in=180] (b);
      \draw (br.south west) to[out=-90,in=0] (b);
      \draw (b.south) to (0,-1.5) node[dot] {};
    \end{pic}
    \end{verticalhack}\qedhere
  \]
\end{proof}

A monoidal dagger category is \emph{positively monoidal} when endomorphisms $f \colon A \to A$ are positive as soon as $f \otimes \id{A}$ is positive. For such categories, one can prove Stinespring's theorem~\cite[Proposition~3.4]{Coecke2013i}, showing that the morphisms $f \colon A \to B$ are precisely those such that $f \otimes \id{E}$ sends completely positive maps $I \to A \otimes E$ to completely positive maps $I \to B \otimes E$ for all special dagger Frobenius structures $E$. It follows that $\CP(\cat{C})$ is equivalent to $\CP^*(\cat{C})$~\cite{Coecke2013i}. In particular, $\CP(\cat{FHilb})$ is the category of finite-dimensional C*-algebras and completely positive maps. However, this is irrelevant to this paper; for more details, see the forthcoming~\cite{heunenkissingervicary:cp}.

\begin{proposition} 
  The category $\RelC$ is positively monoidal for any regular category $\catC$.
\end{proposition}
\begin{proof}
  If a relation $R \colon A \relto A$ satisfies $R \otimes \id{A} = S^\dag \circ S$ for some relation $S \colon A \times A \relto X $, then $R$ is equal to $T^\dag \circ T$ where $T = \llbracket (a,x) \in A \times X \mid (\exists c \in A)\ S(a,c,e) \rrbracket$ and hence is positive.
\end{proof}

\end{document}

%% file: preamble.tex
\usepackage{amsmath,amsthm,amssymb,stmaryrd,xspace,enumerate}
\usepackage{tikz,tikzfig}
\usepackage{breakurl}
\usepackage{url}
\input{tikzstyles.tex}

\newcommand{\sem}[1]{\ensuremath{\llbracket #1 \rrbracket}}

\newcommand{\cat}[1]{\ensuremath{\mathbf{#1}}\xspace}
\newcommand{\FHilb}{\cat{FHilb}}

\newcommand{\CP}{\ensuremath{\cat{CP}\xspace}}
\newcommand{\CPs}{\ensuremath{\CP^*}\xspace}

\newcommand{\defined}{\ensuremath{\mathop{\downarrow}}}

\theoremstyle{definition}
\newtheorem{theorem}{Theorem}[section]
\newtheorem{corollary}[theorem]{Corollary}
\newtheorem{lemma}[theorem]{Lemma}
\newtheorem{proposition}[theorem]{Proposition}
\theoremstyle{definition}
\newtheorem{definition}[theorem]{Definition}
\newtheorem{example}[theorem]{Example}

\usetikzlibrary{calc}

\usetikzlibrary{shapes}
\tikzstyle{dot}=[circle, draw=black, fill=black!25, inner sep=.4ex]
\tikzstyle{black_dot}=[dot, fill=black!50]
\tikzstyle{white_dot}=[dot, fill=white]

\newif\ifvflip\pgfkeys{/tikz/vflip/.is if=vflip}
\newif\ifhflip\pgfkeys{/tikz/hflip/.is if=hflip}
\newif\ifhvflip\pgfkeys{/tikz/hvflip/.is if=hvflip}
\newlength\morphismheight
\newlength\wedgewidth
\tikzset{width/.initial=1mm}
\makeatletter
\tikzstyle{morphism}=[font=\small,morphismshape]
\pgfdeclareshape{morphismshape}
{
    \savedanchor\centerpoint
    {
        \pgf@x=0pt
        \pgf@y=0pt
    }
    \anchor{center}{\centerpoint}
    \anchorborder{\centerpoint}
    \saveddimen\overallwidth
    {
        \pgfkeysgetvalue{/tikz/width}{\minwidth}
        \pgf@x=\wd\pgfnodeparttextbox
        \ifdim\pgf@x<\minwidth
            \pgf@x=\minwidth
        \fi
    }
    \savedanchor{\upperrightcorner}
    {
        \pgf@y=.5\ht\pgfnodeparttextbox
        \advance\pgf@y by -.5\dp\pgfnodeparttextbox
        \pgf@x=.5\wd\pgfnodeparttextbox
    }
    \anchor{north}
    {
        \pgf@x=0pt
        \pgf@y=0.5\morphismheight
    }
    \anchor{north east}
    {
        \pgf@x=\overallwidth
        \multiply \pgf@x by 2
        \divide \pgf@x by 5
        \pgf@y=0.5\morphismheight
    }
    \anchor{east}
    {
        \pgf@x=\overallwidth
        \divide \pgf@x by 2
        \advance \pgf@x by 5pt
        \pgf@y=0pt
    }
    \anchor{west}
    {
        \pgf@x=-\overallwidth
        \divide \pgf@x by 2
        \advance \pgf@x by -5pt
        \pgf@y=0pt
    }
    \anchor{north west}
    {
        \pgf@x=-\overallwidth
        \multiply \pgf@x by 2
        \divide \pgf@x by 5
        \pgf@y=0.5\morphismheight
    }
    \anchor{south east}
    {
        \pgf@x=\overallwidth
        \multiply \pgf@x by 2
        \divide \pgf@x by 5
        \pgf@y=-0.5\morphismheight
    }
    \anchor{south west}
    {
        \pgf@x=-\overallwidth
        \multiply \pgf@x by 2
        \divide \pgf@x by 5
        \pgf@y=-0.5\morphismheight
    }
    \anchor{south}
    {
        \pgf@x=0pt
        \pgf@y=-0.5\morphismheight
    }
    \anchor{text}
    {
        \upperrightcorner
        \pgf@x=-\pgf@x
        \pgf@y=-\pgf@y
    }
    \backgroundpath
    {
    \begin{scope}
        \pgfkeysgetvalue{/tikz/fill}{\morphismfill}
        \pgfsetstrokecolor{black}
        \pgfsetlinewidth{.7pt}
        \begin{scope}
        \pgfsetstrokecolor{black}
        \pgfsetfillcolor{white}
        \pgfsetlinewidth{.7pt}
                \ifhflip
                    \pgftransformyscale{-1}
                \fi
                \ifvflip
                    \pgftransformxscale{-1}
                \fi
                \ifhvflip
                    \pgftransformxscale{-1}
                    \pgftransformyscale{-1}
                \fi
                \pgfpathmoveto{\pgfpoint
                    {-0.5*\overallwidth-5pt}
                    {0.5*\morphismheight}}
                \pgfpathlineto{\pgfpoint
                    {0.5*\overallwidth+5pt}
                    {0.5*\morphismheight}}
                \pgfpathlineto{\pgfpoint
                    {0.5*\overallwidth + \wedgewidth}
                    {-0.5*\morphismheight}}
                \pgfpathlineto{\pgfpoint
                    {-0.5*\overallwidth-5pt}
                    {-0.5*\morphismheight}}
                \pgfpathclose
                \pgfusepath{fill,stroke}
        \end{scope}
    \end{scope}
    }
}
\makeatother

\newcommand{\tinymult}[1][dot]{
\smash{\raisebox{-2pt}{\hspace{-5pt}\ensuremath{\begin{pic}[scale=0.4,yscale=-1]
    \node (0) at (0,0) {};
    \node[#1, inner sep=1.5pt] (1) at (0,0.55) {};
    \node (2) at (-0.5,1) {};
    \node (3) at (0.5,1) {};
    \draw (0.center) to (1.center);
    \draw (1.center) to [out=left, in=down, out looseness=1.5] (2.center);
    \draw (1.center) to [out=right, in=down, out looseness=1.5] (3.center);
    \node[#1, inner sep=1.5pt] (1) at (0,0.55) {};
\end{pic}
}\hspace{-3pt}}}}

\newcommand{\tinymultflip}[1][dot]{
\smash{\raisebox{-2pt}{\hspace{-5pt}\ensuremath{\begin{pic}[scale=0.4,yscale=1]
    \node (0) at (0,0) {};
    \node[#1, inner sep=1.5pt] (1) at (0,0.55) {};
    \node (2) at (-0.5,1) {};
    \node (3) at (0.5,1) {};
    \draw (0.center) to (1.center);
    \draw (1.center) to [out=left, in=down, out looseness=1.5] (2.center);
    \draw (1.center) to [out=right, in=down, out looseness=1.5] (3.center);
    \node[#1, inner sep=1.5pt] (1) at (0,0.55) {};
\end{pic}
}\hspace{-3pt}}}}

\newcommand{\tinyunit}[1][dot]{
\smash{\raisebox{1pt}{\hspace{-3pt}\ensuremath{\begin{pic}[scale=0.4,yscale=-1]
    \node (0) at (0,0) {};
    \node[#1, inner sep=1.5pt] (1) at (0,0.55) {};
    \draw (0.center) to (1.north);
\end{pic}
}\hspace{-1pt}}}}

\newcommand\relto[1][]{\mathbin{\smash{
\begin{tikzpicture}[baseline={([yshift=-1pt]
current bounding box.south)}]
    \node (A) at (0,0) [inner xsep=0pt, inner ysep=1pt, minimum width=0.15cm] {\ensuremath{\scriptstyle #1}};
    \draw [->, line width=0.4pt, line cap=round]
        ([xshift=-2.5pt] A.south west)
        to ([xshift=3pt] A.south east);
    \draw [line width=.4pt] ([xshift=-1pt,yshift=-2pt]A.south) to ([xshift=1pt,yshift=2pt]A.south);
\end{tikzpicture}}}}

\newcommand\sxto[1]{\mathbin{\smash{
\begin{tikzpicture}[baseline={([yshift=-1pt]
current bounding box.south)}]
    \node (A) at (0,0) [inner xsep=0pt, inner ysep=1pt, minimum width=0.15cm] {\ensuremath{\scriptstyle #1}};
    \draw [->, line width=0.4pt, line cap=round]
        ([xshift=-2.5pt] A.south west)
        to ([xshift=3pt] A.south east);
\end{tikzpicture}}}}

\newenvironment{pic}[1][]
{\begin{aligned}\begin{tikzpicture}[font=\tiny,#1]}
{\end{tikzpicture}\end{aligned}}

\newcommand{\id}[1]{\ensuremath{\mathrm{id}_{#1}}}

\newcommand{\catC}{\cat{C}}
\newcommand{\Rel}{\ensuremath{\cat{Rel}}}
\newcommand{\RelC}{\ensuremath{\Rel(\cat{C})}}

%% file: tikzstyles.tex
\tikzstyle{dotpic}=[scale=0.5]

\tikzstyle{braceedge}=[decorate,decoration={brace,amplitude=1mm,raise=-1mm}]
\tikzstyle{left hook arrow}=[left hook-latex]
\tikzstyle{right hook arrow}=[right hook-latex]

\tikzstyle{dot}=[inner sep=0.7mm,minimum width=0pt,minimum height=0pt,fill=black,draw=black,shape=circle]

\tikzstyle{black dot}=[dot]
\tikzstyle{white dot}=[dot,fill=white]
\tikzstyle{gray dot}=[dot,fill=gray!40!white]

\tikzstyle{alt white dot}=[white dot,label={[xshift=3mm,yshift=-0.05mm,font=\tiny]left:$*$}]
\tikzstyle{alt gray dot}=[gray dot,label={[xshift=3mm,yshift=-0.05mm,font=\tiny]left:$*$}]

\tikzstyle{white norm}=[rectangle,fill=white,draw=black,minimum height=2mm,minimum width=2mm,inner sep=0pt,font=\small]
\tikzstyle{gray norm}=[white norm,fill=gray!40!white]
\tikzstyle{black norm}=[white norm,fill=black]

\tikzstyle{arrs}=[-latex,font=\small,auto]
\tikzstyle{arrow plain}=[arrs]
\tikzstyle{arrow dashed}=[dashed,arrs]
\tikzstyle{arrow bold}=[very thick,arrs]
\tikzstyle{arrow hide}=[draw=white!0,-]
\tikzstyle{arrow reverse}=[latex-]
\tikzstyle{cdnode}=[]

\tikzstyle{wide point}=[fill=white,draw=black,shape=isosceles triangle,shape border rotate=90,isosceles triangle stretches=true,inner sep=1pt,minimum width=1.5cm,minimum height=5mm]
\tikzstyle{wide copoint}=[fill=white,draw=black,shape=isosceles triangle,shape border rotate=-90,isosceles triangle stretches=true,inner sep=1pt,minimum width=1.5cm,minimum height=4mm]
\tikzstyle{very wide copoint}=[fill=white,draw=black,shape=isosceles triangle,Shape border rotate=-90,isosceles triangle stretches=true,inner sep=1pt,minimum width=2.5cm,minimum height=4mm]
\tikzstyle{very wide empty copoint}=[draw=black,shape=isosceles triangle,shape border rotate=-90,isosceles triangle stretches=true,inner sep=1pt,minimum width=2.5cm,minimum height=4mm]
\tikzstyle{symm}=[ultra thick,shorten <=-1mm,shorten >=-1mm]

\tikzstyle{square box}=[rectangle,fill=white,draw=black,minimum height=5mm,minimum width=5mm,font=\small]
\tikzstyle{square gray box}=[rectangle,fill=gray!30,draw=black,minimum height=6mm,minimum width=6mm]
\tikzstyle{point}=[regular polygon,regular polygon rotate=180,regular polygon sides=3,draw=black,scale=0.75,inner sep=-0.5pt,minimum width=.6cm,fill=white]
\tikzstyle{copoint}=[regular polygon,regular polygon sides=3,draw=black,scale=0.75,inner sep=-0.5pt,minimum width=1cm,fill=white]
\tikzstyle{gray point}=[point,fill=gray!40!white]
\tikzstyle{gray copoint}=[copoint,fill=gray!40!white]

\tikzstyle{diredge}=[->]
\tikzstyle{rdiredge}=[<-]
\tikzstyle{dashed edge}=[dashed]

\tikzstyle{cross}=[preaction={draw=white, -, line width=3pt}]

\tikzstyle{black point}=[regular polygon,regular polygon rotate=180,regular polygon sides=3,draw=black,scale=0.25,inner sep=-0.5pt,minimum width=1cm,fill=black]

\tikzstyle{grey point}=[regular polygon,regular polygon rotate=180,regular polygon sides=3,draw=black,scale=0.25,inner sep=-0.5pt,minimum width=1cm,fill=gray!40!white]

\tikzstyle{white point}=[regular polygon,regular polygon rotate=180,regular polygon sides=3,draw=black,scale=0.25,inner sep=-0.5pt,minimum width=1cm,fill=white]

\tikzstyle{wide white point}=[regular polygon,regular polygon rotate=180,regular polygon sides=3,draw=black,xscale=0.25,yscale=0.15,inner sep=-0.5pt,minimum width=15mm,fill=white]

\tikzstyle{wide white copoint}=[regular polygon,regular polygon rotate=0,regular polygon sides=3,draw=black,xscale=0.25,yscale=0.15,inner sep=-0.5pt,minimum width=15mm,fill=white]

\tikzstyle{black copoint}=[regular polygon,regular polygon rotate=0,regular polygon sides=3,draw=black,scale=0.25,inner sep=-0.5pt,minimum width=1cm,fill=black]

\tikzstyle{grey copoint}=[regular polygon,regular polygon rotate=0,regular polygon sides=3,draw=black,scale=0.25,inner sep=-0.5pt,minimum width=1cm,fill=gray!40!white]

\tikzstyle{white copoint}=[regular polygon,regular polygon rotate=0,regular polygon sides=3,draw=black,scale=0.25,inner sep=-0.5pt,minimum width=1cm,fill=white]

\newcommand{\subsystemcolour}{gray!40!white}
\pgfkeys{/tikz/.cd, subsystemcolour/.store in=\subsystemcolour}
\tikzstyle{subsystem}=[postaction={decorate,decoration={markings,
    mark=at position .4 with {\arrow[rotate=180, fill=\subsystemcolour]{triangle 60}}}}]
\tikzstyle{None}=[circle, fill=white, inner sep=0pt]